\documentclass[a4paper,10pt]{article}
\usepackage{amsthm,amssymb}
\usepackage{mathtools}
\usepackage[hidelinks]{hyperref}
\usepackage[a4paper,width=160mm,top=27mm,bottom=27mm]{geometry}

\newtheorem{theorem}{Theorem}[section]
\newtheorem{lemma}[theorem]{Lemma}
\newtheorem{definition}[theorem]{Definiton}
\newtheorem{proposition}[theorem]{Proposition}
\newtheorem{corollary}[theorem]{Corollary}

\theoremstyle{definition}
\newtheorem{remx}[theorem]{Remark}
\newenvironment{remark}
  {\pushQED{\qed}\remx}
  {\popQED\endremx}

\makeatletter
\newsavebox\myboxA
\newsavebox\myboxB
\newlength\mylenA

\newcommand*\yoverline[2][0.75]{%
    \sbox{\myboxA}{$\m@th#2$}%
    \setbox\myboxB\null
    \ht\myboxB=\ht\myboxA%
    \dp\myboxB=\dp\myboxA%
    \wd\myboxB=#1\wd\myboxA
    \sbox\myboxB{$\m@th\overline{\copy\myboxB}$}
    \setlength\mylenA{\the\wd\myboxA}
    \addtolength\mylenA{-\the\wd\myboxB}%
    \ifdim\wd\myboxB<\wd\myboxA%
       \rlap{\hskip 0.5\mylenA\usebox\myboxB}{\usebox\myboxA}%
    \else
        \hskip -0.5\mylenA\rlap{\usebox\myboxA}{\hskip 0.5\mylenA\usebox\myboxB}%
    \fi}
\makeatother

\numberwithin{equation}{section}

\begin{document}





\newcommand{\diver}{\operatorname{div}}
\newcommand{\lin}{\operatorname{Lin}}
\newcommand{\curl}{\operatorname{curl}}
\newcommand{\ran}{\operatorname{Ran}}
\newcommand{\kernel}{\operatorname{Ker}}
\newcommand{\la}{\langle}
\newcommand{\ra}{\rangle}
\newcommand{\N}{\mathbb{N}}
\newcommand{\R}{\mathbb{R}}
\newcommand{\C}{\mathbb{C}}
\newcommand{\T}{\mathbb{T}}

\newcommand{\ld}{\lambda}
\newcommand{\fai}{\varphi}
\newcommand{\0}{0}
\newcommand{\n}{\mathbf{n}}
\newcommand{\uu}{{\boldsymbol{\mathrm{u}}}}
\newcommand{\UU}{{\boldsymbol{\mathrm{U}}}}
\newcommand{\buu}{\bar{{\boldsymbol{\mathrm{u}}}}}
\newcommand{\ten}{\\[4pt]}
\newcommand{\six}{\\[-3pt]}
\newcommand{\nb}{\nonumber}
\newcommand{\hgamma}{H_{\Gamma}^1(\OO)}
\newcommand{\opert}{O_{\varepsilon,h}}
\newcommand{\barx}{\bar{x}}
\newcommand{\barf}{\bar{f}}
\newcommand{\hatf}{\hat{f}}
\newcommand{\xoneeps}{x_1^{\varepsilon}}
\newcommand{\xh}{x_h}
\newcommand{\scaled}{\nabla_{1,h}}
\newcommand{\scaledb}{\widehat{\nabla}_{1,\gamma}}
\newcommand{\vare}{\varepsilon}
\newcommand{\A}{{\bf{A}}}
\newcommand{\RR}{{\bf{R}}}
\newcommand{\B}{{\bf{B}}}
\newcommand{\CC}{{\bf{C}}}
\newcommand{\D}{{\bf{D}}}
\newcommand{\K}{{\bf{K}}}
\newcommand{\oo}{{\bf{o}}}
\newcommand{\id}{{\bf{Id}}}
\newcommand{\E}{\mathcal{E}}
\newcommand{\ii}{\mathcal{I}}
\newcommand{\sym}{\mathrm{sym}}
\newcommand{\lt}{\left}
\newcommand{\rt}{\right}
\newcommand{\ro}{{\bf{r}}}
\newcommand{\so}{{\bf{s}}}
\newcommand{\e}{{\bf{e}}}
\newcommand{\ww}{{\boldsymbol{\mathrm{w}}}}
\newcommand{\zz}{{\boldsymbol{\mathrm{z}}}}
\newcommand{\U}{{\boldsymbol{\mathrm{U}}}}
\newcommand{\G}{{\boldsymbol{\mathrm{G}}}}
\newcommand{\VV}{{\boldsymbol{\mathrm{V}}}}
\newcommand{\II}{{\boldsymbol{\mathrm{I}}}}
\newcommand{\ZZ}{{\boldsymbol{\mathrm{Z}}}}
\newcommand{\hKK}{{{\bf{K}}}}
\newcommand{\f}{{\bf{f}}}
\newcommand{\g}{{\bf{g}}}
\newcommand{\lkk}{{\bf{k}}}
\newcommand{\tkk}{{\tilde{\bf{k}}}}
\newcommand{\W}{{\boldsymbol{\mathrm{W}}}}
\newcommand{\Y}{{\boldsymbol{\mathrm{Y}}}}
\newcommand{\EE}{{\boldsymbol{\mathrm{E}}}}
\newcommand{\F}{{\bf{F}}}
\newcommand{\spacev}{\mathcal{V}}
\newcommand{\spacevg}{\mathcal{V}^{\gamma}(\Omega\times S)}
\newcommand{\spacevb}{\bar{\mathcal{V}}^{\gamma}(\Omega\times S)}
\newcommand{\spaces}{\mathcal{S}}
\newcommand{\spacesg}{\mathcal{S}^{\gamma}(\Omega\times S)}
\newcommand{\spacesb}{\bar{\mathcal{S}}^{\gamma}(\Omega\times S)}
\newcommand{\skews}{H^1_{\barx,\mathrm{skew}}}
\newcommand{\kk}{\mathcal{K}}
\newcommand{\OO}{O}
\newcommand{\bhe}{{\bf{B}}_{\vare,h}}
\newcommand{\pp}{{\mathbb{P}}}
\newcommand{\ff}{{\mathcal{F}}}
\newcommand{\mWk}{{\mathcal{W}}^{k,2}(\Omega)}
\newcommand{\mWa}{{\mathcal{W}}^{1,2}(\Omega)}
\newcommand{\mWb}{{\mathcal{W}}^{2,2}(\Omega)}
\newcommand{\twos}{\xrightharpoonup{2}}
\newcommand{\twoss}{\xrightarrow{2}}
\newcommand{\bw}{\bar{w}}
\newcommand{\bz}{\bar{{\bf{z}}}}
\newcommand{\tw}{{W}}
\newcommand{\tr}{{{\bf{R}}}}
\newcommand{\tz}{{{\bf{Z}}}}
\newcommand{\lo}{{{\bf{o}}}}
\newcommand{\hoo}{H^1_{00}(0,L)}
\newcommand{\ho}{H^1_{0}(0,L)}
\newcommand{\hotwo}{H^1_{0}(0,L;\R^2)}
\newcommand{\hooo}{H^1_{00}(0,L;\R^2)}
\newcommand{\hhooo}{H^1_{00}(0,1;\R^2)}
\newcommand{\dsp}{d_{S}^{\bot}(\barx)}
\newcommand{\LB}{{\bf{\Lambda}}}
\newcommand{\LL}{\mathbb{L}}
\newcommand{\mL}{\mathcal{L}}
\newcommand{\mhL}{\widehat{\mathcal{L}}}
\newcommand{\loc}{\mathrm{loc}}
\newcommand{\tqq}{\mathcal{Q}^{*}}
\newcommand{\tii}{\mathcal{I}^{*}}
\newcommand{\Mts}{\mathbb{M}}
\newcommand{\pot}{\mathrm{pot}}
\newcommand{\tU}{{\widehat{\bf{U}}}}
\newcommand{\tVV}{{\widehat{\bf{V}}}}
\newcommand{\pt}{\partial}
\newcommand{\bg}{\Big}
\newcommand{\hA}{\widehat{{\bf{A}}}}
\newcommand{\hB}{\widehat{{\bf{B}}}}
\newcommand{\hCC}{\widehat{{\bf{C}}}}
\newcommand{\hD}{\widehat{{\bf{D}}}}
\newcommand{\fder}{\partial^{\mathrm{MD}}}
\newcommand{\Var}{\mathrm{Var}}
\newcommand{\pta}{\partial^{0\bot}}
\newcommand{\ptaj}{(\partial^{0\bot})^*}
\newcommand{\ptb}{\partial^{1\bot}}
\newcommand{\ptbj}{(\partial^{1\bot})^*}
\newcommand{\geg}{\Lambda_\vare}
\newcommand{\tpta}{\tilde{\partial}^{0\bot}}
\newcommand{\tptb}{\tilde{\partial}^{1\bot}}
\newcommand{\ua}{u_\alpha}
\newcommand{\pa}{p\alpha}
\newcommand{\qa}{q(1-\alpha)}
\newcommand{\Qa}{Q_\alpha}
\newcommand{\Qb}{Q_\eta}
\newcommand{\ga}{\gamma_\alpha}
\newcommand{\gb}{\gamma_\eta}
\newcommand{\ta}{\theta_\alpha}
\newcommand{\tb}{\theta_\eta}


\newcommand{\mH}{{E}}
\newcommand{\mN}{{N}}
\newcommand{\mD}{{\mathcal{D}}}
\newcommand{\csob}{\mathcal{S}}
\newcommand{\mA}{{A}}
\newcommand{\mK}{{Q}}
\newcommand{\mS}{{S}}
\newcommand{\mI}{{I}}
\newcommand{\tas}{{2_*}}
\newcommand{\tbs}{{2^*}}
\newcommand{\tm}{{\tilde{m}}}
\newcommand{\tdu}{{\phi}}
\newcommand{\tpsi}{{\tilde{\psi}}}
\newcommand{\Z}{{\mathbb{Z}}}
\newcommand{\tsigma}{{\tilde{\sigma}}}
\newcommand{\tg}{{\tilde{g}}}
\newcommand{\tG}{{\tilde{G}}}
\newcommand{\mM}{{M}}
\newcommand{\mC}{\mathcal{C}}
\newcommand{\wlim}{{\text{w-lim}}\,}
\newcommand{\diag}{L_t^\ba L_x^\br}
\newcommand{\vu}{ u}
\newcommand{\vz}{ z}
\newcommand{\vv}{ v}
\newcommand{\ve}{ e}
\newcommand{\vw}{ w}
\newcommand{\vf}{ f}
\newcommand{\vh}{ h}
\newcommand{\vp}{ \vec P}
\newcommand{\ang}{{\not\negmedspace\nabla}}
\newcommand{\dxy}{\Delta_{x,y}}
\newcommand{\lxy}{L_{x,y}}
\newcommand{\gnsand}{\mathrm{C}_{\mathrm{GN},3d}}
\newcommand{\wmM}{\widehat{{M}}}
\newcommand{\wmH}{\widehat{{E}}}
\newcommand{\wmI}{\widehat{{I}}}
\newcommand{\wmK}{\widehat{{Q}}}
\newcommand{\wmN}{\widehat{{N}}}
\newcommand{\wm}{\widehat{m}}
\newcommand{\ba}{\mathbf{a}}
\newcommand{\bb}{\mathbf{b}}
\newcommand{\br}{\mathbf{r}}
\newcommand{\bq}{\mathbf{q}}
\newcommand{\SSS}{\mathcal{S}}

\title{On long time behavior of the focusing energy-critical NLS on $\R^d\times\T$ via semivirial-vanishing geometry}
\author{Yongming Luo \thanks{Institut f\"{u}r Wissenschaftliches Rechnen, Technische Universit\"at Dresden, Germany} \thanks{\href{mailto:yongming.luo@tu-dresden.de}{Email: yongming.luo@tu-dresden.de}}
}

\date{}
\maketitle

\begin{abstract}
We study the focusing energy-critical NLS
\begin{align}\label{nls_abstract}
i\pt_t u+\Delta_{x,y} u=-|u|^{\frac{4}{d-1}} u\tag{NLS}
\end{align}
on the waveguide manifold $\R_x^d\times\T_y$ with $d\geq 2$. We reveal the somewhat counterintuitive phenomenon that despite the energy-criticality of the nonlinear potential, the long time dynamics of \eqref{nls_abstract} are purely determined by the semivirial-vanishing geometry which possesses an \textit{energy-subcritical} characteristic. As a starting point, we consider a minimization problem $m_c$ defined on the semivirial-vanishing manifold with prescribed mass $c$. We prove that for all sufficiently large mass the variational problem $m_c$ has a unique optimizer $u_c$ satisfying $\pt_y u_c=0$, while for all sufficiently small mass, any optimizer of $m_c$ must have non-trivial $y$-dependence. Afterwards, we prove that $m_c$ characterizes a sharp threshold for the bifurcation of finite time blow-up ($d=2,3$) and globally scattering ($d=3$) solutions of \eqref{nls_abstract} in dependence of the sign of the semivirial. To the author's knowledge, the paper also gives the first large data scattering result for focusing NLS on product spaces in the energy-critical setting.
\end{abstract}


\section{Introduction and main results}\label{sec:Introduction and main results}
We study the focusing energy-critical nonlinear Schr\"odinger equation (NLS)
\begin{align}\label{nls}
i\pt_t u+\Delta_{x,y} u=-|u|^{\frac{4}{d-1}} u
\end{align}
on the waveguide manifold $\R_x^d\times \T_y$ with $d\geq 2$ and $\T=\R/2\pi\Z$. The equation \eqref{nls} serves as a toy model in various physical applications such as nonlinear optics and Bose-Einstein condensation. For a more comprehensive introduction on the physical background of \eqref{nls}, we refer to \cite{waveguide_ref_1,waveguide_ref_2,waveguide_ref_3} and the references therein. From a mathematical point of view, the mixed type nature of the underlying domain also makes the analysis of \eqref{nls} rather challenging and interesting. In a previous paper \cite{Luo_inter} the so-called \textit{semivirial-vanishing geometry} has been introduced by the author to study the intercritical analogue of \eqref{nls}. The purpose of this paper is to reveal the interesting phenomenon that despite its \textit{energy-subcritical} characteristic, the framework of semivirial-vanishing geometry indeed continues to work for the energy-critical model \eqref{nls}.

In recent years, there has been an increasing interest in studying dispersive equations on compact manifolds and product spaces. Among all, we underline that the first result for energy-critical NLS on tori might date back to Herr, Tataru and Tzvetkov \cite{HerrTataruTz1}, where the local well-posedness of the quintic NLS on $\T^3$ was shown. As an application, the framework of $X^s$- and $Y^s$-spaces developed in \cite{HerrTataruTz1} was later invoked to obtain local well-posedness results for energy-critical NLS on other different manifolds such as the 4D product space $\R^d\times\T^{4-d}$ \cite{HerrTataruTz2} and Zoll manifolds \cite{HerrZoll}. By appealing to the concentration compactness arguments initiated by Kenig and Merle \cite{KenigMerle2006} and using the global well-posedness results for the defocusing energy-critical NLS on $\R^3$ and $\R^4$ \cite{defocusing3d,defocusing4d} as Black-Box-Theories, Ionescu, Pausader and Staffilani showed that the defocusing energy-critical NLS is alwalys globally well-posed on $\T^3$, $\R\times\T^3$ and on the three-dimensional hyperbolic space $\mathbb{H}^3$ \cite{Ionescu1,Ionescu2,hyperbolic}. Following the same ideas in \cite{Ionescu1,Ionescu2,hyperbolic}, a corresponding large data well-posedness result for focusing energy-critical NLS on $\T^4$ and $\R\times\T^3$ has been recently established by Yu, Yue and Zhao \cite{Haitian,YuYueZhao2021}. We also refer to \cite{BrezisGallouet1980,Bourgain1,Bourgain2,Burq1,Burq2,Burq3,ZhaoZheng2021} for further interesting results in this direction.

Although the large data well-posedness results are already satisfactory to certain extent, we are more interested in results concerning scattering or finite time blow-up phenomena which give a more accurate description on the long time dynamics of a solution. In general, however, the scattering results are more difficult to prove, as they demand a global control on the decay of the solutions. Moreover, while a solution of an NLS on $\R^d$ with nonlinearity of intercritical growth shall be scattering in time in the energy space under certain circumstances (e.g. small initial data), a solution on $\T^n$ does not scatter principally. The situation thus becomes more interesting when considering an NLS on the product space $\R^d\times\T^n$. We naturally ask whether the strong dispersion coming from the $\R^d$-side could ultimately guarantee the scattering of a solution of NLS on $\R^d\times\T^n$. Motivated by the scattering results of NLS on $\R^d$, we expect a solution to be scattering as long as
\begin{itemize}
\item[(i)] The nonlinearity is at most energy-critical w.r.t. the space dimension $d+n$.
\item[(ii)]The nonlinearity is at least mass-critical w.r.t. the space dimension $d$.
\end{itemize}
This particularly requires $n\in\{1,2\}$ (the case $n=0$ reduces to the well-known $\R^d$-case). The case $n=2$ is significantly more interesting and difficult than the case $n=1$, as in the case $n=2$ the nonlinearity is both mass- and energy-critical. The first breakthrough in this direction was made by Hani and Pausader \cite{HaniPausader}, where the authors studied the defocusing quintic NLS on the waveguide manifold $\R\times\T^2$. In particular, the authors proved suitable Strichartz estimates on $\R\times\T^2$ which are given in terms of the $X^s$- and $Y^s$-spaces introduced in \cite{HerrTataruTz1,HerrTataruTz2} and particularly provide sufficiently strong dispersion that guarantee the scattering of a solution in the energy space with small initial data. Indeed, by making use of the Black-Box-Theory initiated in \cite{Ionescu1,Ionescu2,hyperbolic} the authors were also able to prove that a solution of the defocusing quintic NLS on $\R\times\T^2$ is always global and scattering. We shall also point out that the large data scattering result in \cite{HaniPausader} was originally conditional and based on a conjecture concerning the large data scattering result of the corresponding large scale resonant system, which was later confirmed by \cite{R1T1Scattering}. By making use of the idea from \cite{HaniPausader} the large data scattering problems for defocusing NLS on waveguide manifolds with algebraic nonlinearities have been completely resolved, see \cite{HaniPausader,Yang_Zhao_2018,R1T1Scattering,CubicR2T1Scattering,RmT1,R2T2}. We shall also underline the interesting paper by Tzvetkov and Visciglia \cite{TzvetkovVisciglia2016}, where the defocusing intercritical analogue of \eqref{nls} on $\R^d\times\T$ was investigated. Instead of using the concentration compactness principle, the large data scattering result given in \cite{TzvetkovVisciglia2016} was proved by using the interaction Morawetz inequality originated in \cite{defocusing3d}. In particular, the scattering result given in \cite{TzvetkovVisciglia2016} is available for all $d\geq 1$ and all intercritical nonlinearities which are not necessarily algebraic. We also refer to \cite{TNCommPDE,FoprcellaHari2020,SystemProdSpace,ModifiedScattering,Barron,BarronChristPausader2021,Cheng_JMAA,ZhaoZheng2021,YuYueZhao2021,4NLS} for further interesting works on scattering problems of dispersive equations on product spaces.

Despite the abundant results on the defocusing NLS on product spaces, similar large data scattering results for the focusing model are less well-known. The first result\footnote{We shall point out that the large data scattering result of the corresponding large scale resonant system proved in \cite{Luo_Waveguide_MassCritical} was independently shown in \cite{similar_cubic}.} on large data scattering for focusing NLS on product space was recently given by the author \cite{Luo_Waveguide_MassCritical}, where the focusing cubic NLS on $\R^2\times\T$ was studied. A key new ingredient for showing the large data scattering result in \cite{Luo_Waveguide_MassCritical} was a specialized Gagliardo-Nirenberg inequality on $\R^2\times\T$ which is particularly scale-invariant w.r.t. the $x$-variable and differs from the standard inhomogeneous ones. Nonetheless, the scattering threshold formulated in \cite{Luo_Waveguide_MassCritical} might possibly be non-sharp, since we were unable to prove ground state solutions lying on the threshold and a criterium for a solution being blowing-up was also missing. The situation becomes much better when the nonlinearity is at least mass-supercritical. By introducing the so-called \textit{semivirial-vanishing geometry} \cite{Luo_inter} the author was able to formulate a threshold which determines the bifurcation of finite time blow-up and globally scattering solutions of the focusing intercritical NLS on $\R^d\times\T$ in dependence of the sign of the semivirial functional. Moreover, the given threshold is sharp due to the existence of ground state solutions lying on the threshold. Using a subtle scaling argument according to Terracini-Tzvetkov-Visciglia \cite{TTVproduct2014} we were also able to prove the existence of a critical value $c_*\in(0,\infty)$ which sharply determines the $y$-dependence of the ground state solutions in dependence of the size of the mass $c$.

In this paper we continue our study on the focusing energy-critical NLS \eqref{nls} by appealing to the framework of semivirial-vanishing geometry. As a starting point, we shall firstly introduce some basic concepts of the underlying theory.

\subsubsection*{Semivirial-vanishing geometry}
When considering an NLS on $\R^d$, a very important tool to study the long time dynamics of a solution is the celebrated Glassey's virial identity. Consider for instance a solution $u$ of the NLS
\begin{align}\label{example}
i\pt_t u+\Delta_x u=-|u|^\alpha u\quad\text{on $\R^d$}.
\end{align}
We then define the virial action functional $V_{\R^d}(t)$ by
$$V_{\R^d}(t):=\int_{\R^d}|x|^2|u(t,x)|^2\,dx.$$
In \cite{Glassey1977}, Glassey gave the following celebrated virial identity
$$ \frac{d^2}{dt^2}V_{\R^d}(t)=8\wmK(u(t)):=8\bg(\|\nabla_x u(t)\|_{L^2(\R^d)}^2-\frac{\alpha d}{2(\alpha+2)}\|u(t)\|_{L^{\alpha+2}(\R^d)}^{\alpha+2}\bg).$$
The quantity $\wmK(u)$ is usually referred to as the virial of the solution $u$. Hence a virial bounded above by some negative number shall ultimately lead to a finite time blow-up. On the contrary, a positive virial might be a sign indicating global well-posedness or even scattering of a solution. In fact, this can be rigorously proved under certain circumstances, see for instance \cite{weinstein,KenigMerle2006}. Motivated by these heuristics, when considering an NLS on $\R^d\times\T$ it is therefore natural to define the similar quantity $\int_{\R^d\times\T}|(x,y)|^2|u(t,x,y)|^2\,dxdy$. However, due to the boundedness of $\T$ the previously defined quantity is in fact less helpful for obtaining results concerning the long time dynamics of the NLS. Indeed, by calculating the second time derivative explicitly we shall see that there are some boundary integral terms remaining which can not be eliminated even by invoking the periodicity of the solution. Alternatively, we shall consider the quantity
$$V_{\R^d\times\T}(t):=\int_{\R^d\times\T}|x|^2|u(t,x,y)|^2\,dxdy.$$
By formally taking the second time derivative we arrive at
$$ \frac{d^2}{dt^2}V_{\R^d\times\T}(t)=8\mK(u(t)):=8\bg(\|\nabla_x u(t)\|_{L^2(\R^d\times\T)}^2-\frac{\alpha d}{2(\alpha+2)}\|u(t)\|_{L^{\alpha+2}(\R^d\times\T)}^{\alpha+2}\bg).$$
We shall simply refer $\mK(u)$ to as the \textit{semivirial} functional. At the first glance, the way we define $\mK(u)$ is purely due to the issue that $\T$ is bounded. However, since generally we do not expect a nonlinear Schr\"odiger wave to be scattering along the torus side, it seems reasonable to consider the dispersive effects that are purely provided by the $\R^d$-side. Nevertheless, from Theorem \ref{thm threshold mass} given below we shall see that albeit such heuristics might be true for solutions with large mass, in the small mass case the impact from the torus side must be taken into account in a non-trivial way. From now on we focus on the problem \eqref{nls} and set
$$ \mK(u)=\|\nabla_x u\|_{L^2(\R^d\times\T)}^2-\frac{d}{d+1}\|u\|_{L^{2+4/(d-1)}(\R^d\times\T)}^{2+4/(d-1)}.$$
Motivated by Jeanjean's seminal work \cite{Jeanjean1997} we define the variational problem $m_c$ on the semivirial-vanishing manifold with prescribed mass $c$ by
\begin{align}
m_c:=\inf_{u\in H^1(\R^d\times\T)}\{\mH(u):\mM(u)=c,\,\mK(u)=0\},\label{mc first def}
\end{align}
where $\mM(u)$ and $\mH(u)$ denote the usual mass and energy of a function $u$:
\begin{align*}
\mM(u):=\|u\|_{L^2(\R^d\times\T)}^2,\quad
\mH(u):=\frac{1}{2}\|\nabla_{x,y}u\|_{L^2(\R^d\times\T)}^2
-\frac{d-1}{2(d+1)}\|u\|_{L^{2+4/(d-1)}(\R^d\times\T)}^{2+4/(d-1)}.
\end{align*}

\subsubsection*{Main results}
In order to formulate our main results, we firstly fix some notation. For a function $u\in H^1(\R^d)$, let the quantities $\wmM(u)$, $\wmH(u)$ and $\wmK(u)$ be the mass, energy and virial of $u$ defined by \eqref{def of wmM}, \eqref{def of wmH} and \eqref{def of wmK} below respectively. For $c\in(0,\infty)$ define also the variational problem $\wm_c$ by
\begin{align}
\wm_c:=\inf_{u\in H^1(\R^d)}\{\wmH(u):\wmM(u)=c,\,\wmK(u)=0\}.\label{wmc first def}
\end{align}
Our first main result reveals the relation between the quantities $m_c$ and $\wm_c$.
\begin{theorem}[$y$-dependence of the ground states]\label{thm threshold mass}
Let $m_c$ and $\wm_c$ be the quantities defined by \eqref{mc first def} and \eqref{wmc first def} respectively. Then there exists some $c_*\in(0,\infty)$ such that
\begin{itemize}
\item[(i)] For all $c\in(0,c_*)$ we have $m_{c}<2\pi \wm_{(2\pi)^{-1}c}$.
\item[(ii)] For all $c\in[c_*,\infty)$ we have $m_{c}=2\pi \wm_{(2\pi)^{-1}c}$. Moreover, if $c\in(c_*,\infty)$, then any minimizer $u_c$ of $m_{c}$ must satisfy $\pt_y u_c=0$.
\end{itemize}
\end{theorem}
Here follow several comments on Theorem \ref{thm threshold mass}:
\begin{itemize}
\item[(i)]Notice by assuming a function $u\in H^1(\R^d\times\T)$ with $\mK(u)=0$ is independent of $y$ we see that $u\in H^1(\R^d)$ and $\wmK(u)=0$. Consequently we infer that $m_c\leq 2\pi\wm_{(2\pi)^{-1}c}$. In particular, from Theorem \ref{thm threshold mass} (ii) it is immediate that for $c\in(c_*,\infty)$ the optimizers of $m_c$ coincide with the ones of $\wm_{(2\pi)^{-1}c}$. By a classical result from \cite{Hajaiej_Stuart_2004}, the set of optimizers of $\wm_{c}$ is given by
    $$ \mathcal{Z}_c=\{e^{i\theta}P_c(\cdot+y):\theta\in\R,\,y\in\R^d\},$$
where $P_c\in H^1(\R^d)$ is the unique radially symmetric and positive solution of
\begin{align*}
-\Delta_x P_c+\omega_c P_c=P_c^{\frac{d+3}{d-1}}\quad\text{on $\R^d$}
\end{align*}
with $\wmM(P_c)=c$ and some positive $\omega_c>0$ which is uniquely determined by the mass $c$. Hence up to symmetries, the ground states of $m_c$ for $c\in(c_*,\infty)$ are unique and independent of $y$.

\item[(ii)] Later we shall use the quantity $m_c$ to formulate a threshold for determining scattering and finite time blow-up solutions (Theorem \ref{main thm} and \ref{thm blow up}). While Theorem \ref{thm threshold mass} (ii) tells us that in the large mass case the balance point of linear dispersion and nonlinear effect is attained at the $\R^d$-ground states, Theorem \ref{thm threshold mass} (i) indicates the interesting fact that in the small mass case, the threshold must display certain non-trivial $y$-dependence.
\end{itemize}
The proof of Theorem \ref{thm threshold mass} follows the same strategy given in \cite{TTVproduct2014,Luo_inter}: For $\ld>0$ we introduce the rescaled energy $\mH_{\ld}(u)$ defined by \eqref{def modified energy} below and consider the minimization problem
\begin{align*}
m_{1,\ld}:=\inf_{u\in H^1(\R^d\times\T)}\{\mH_{\ld}(u):\mM(u)=1,\,\mK(u)=0\}.
\end{align*}
By appealing to a simple rescaling argument, proving Theorem \ref{thm threshold mass} is essentially equivalent to showing the statements
\begin{align*}
\lim_{\ld\to \infty}m_{1,\ld}=2\pi\wm_{(2\pi)^{-1}},\quad\lim_{\ld\to 0}m_{1,\ld}<2\pi\wm_{(2\pi)^{-1}}.
\end{align*}
While the statement $\lim_{\ld\to 0}m_{1,\ld}<2\pi\wm_{(2\pi)^{-1}}$ can be proved by constructing some special test functions, the proof of the statement $\lim_{\ld\to \infty}m_{1,\ld}=2\pi\wm_{(2\pi)^{-1}}$ relies on some subtle coercivity arguments given by \cite{TTVproduct2014}. We shall also point out that despite the proof of Theorem \ref{thm threshold mass} is almost identical to the one given for \cite[Thm 1.2]{Luo_inter}, we encounter the new difficulty that the concentration compactness arguments given by \cite{TTVproduct2014} for identifying a
non-vanishing weak limit of a minimizing sequence of $m_c$ (resp. $m_{1,\ld}$) are no longer valid in the energy-critical setting. The key observation here is that for any minimizing sequence $(u_n)_n$ with sufficiently large mass (resp. large $\ld$), we are able to prove the fact that the zero Fourier coefficients $m(u_n):=(2\pi)^{-1}\int_{\T}u_n \,dy$ of $u_n$ w.r.t $y$ will be concentrating as $n\to\infty$. Since $m(u_n)$ is independent of $y$, we are then able to invoke the classical concentration compactness arguments applied on $\R^d$ to infer a non-vanishing weak limit of $u_n$, as desired. Again, the proof of the concentration effect of $m(u_n)$ relies on a subtle scale-invariant Gagliardo-Nirenberg inequality on $\R^d\times\T$ which follows the same fashion as the ones given in \cite{Luo_Waveguide_MassCritical,Luo_inter}. It remains an open question whether ground states optimizers of $m_c$ exist for $c\in(0,c_*)$. We underline that the proof of showing the existence of ground states with large mass (Proposition \ref{thm existence of ground state 1}) in fact works equally well for $m_c$ with arbitrary $c\in(0,\infty)$ as long as we know that a minimizing sequence has a non-vanishing weak limit. In general we conjecture that $m_c$ has no ground state optimizers when at least $c$ is sufficiently small.

Finally, we prove that the quantity $m_c$ characterizes a sharp threshold for the bifurcation of scattering and finite time blow-up solutions in dependence of the sign of the semivirial.

\begin{theorem}[Scattering below threshold]\label{main thm}
Let $d=3$ and let $u\in X_{c,\mathrm{loc}}(I)$ be a solution of \eqref{nls} with maximal lifespan $I\subset\R$ containing zero, where the space $X_{c,\mathrm{loc}}(I)$ is defined by \eqref{def of Xcloc}. Assume that
$$\mH(u)<m_{\mM(u)}\quad\text{and}\quad\mK(u(0))>0.$$
Then $u$ is a global and scattering in time solution of \eqref{nls} in the sense that there exist $\phi^\pm\in H^1(\R^3\times\T)$ such that
\begin{align}
\lim_{t\to\pm\infty}\|u(t)-e^{it\Delta_{x,y}}\phi^{\pm}\|_{H^1(\R^3\times\T)}=0.
\end{align}
\end{theorem}

\begin{theorem}[Finite time blow-up below threshold]\label{thm blow up}
Let $d\in\{2,3\}$ and let $u\in X_{c,\mathrm{loc}}(I)$ be a solution of \eqref{nls} with maximal lifespan $I\subset\R$ containing zero. Assume that
$$|x|u(0)\in L^2(\R^d\times\T)\,\wedge\,\mH(u)<m_{\mM(u)}\,\wedge\,\mK(u(0))<0.$$
Then $u$ blows-up in finite time.
\end{theorem}

The proofs of Theorem \ref{main thm} and \ref{thm blow up} rely on the classical concentration compactness arguments initiated by Kenig and Merle \cite{KenigMerle2006} and the virial arguments by Glassey \cite{Glassey1977} respectively. The restriction $d<4$ is due to the fact that at the moment the local well-posedness results and Strichartz estimates for \eqref{nls} in higher dimensions ($d\geq 4$) are still open problems (we encounter the difficulty that the nonlinearity is no longer algebraic in higher dimensions). Moreover, the further restriction $d=3$ given in Theorem \ref{main thm} is attributed to the fact that the celebrated Kenig-Merle large data scattering result for focusing quintic NLS on $\R^3$ \cite{KenigMerle2006} is only known to hold for radial initial data. Nevertheless, Theorem \ref{main thm} extends straightforwardly to the case $d=2$ when the corresponding Black-Box-Theory on $\R^3$ is also available for non-radial initial data, which is widely believed to be true.

We shall also emphasize that the main difficulty for proving Theorem \ref{main thm} is to embed a Euclidean profile $T_n^j\phi^j$ appearing in the linear profile decomposition (Lemma \ref{linear profile}) into the Black-Box of the large data scattering result for focusing cubic NLS on $\R^4$ given by Dodson \cite{Dodson4dfocusing}. To be more precise, we mainly need to solve the following two issues:
\begin{itemize}
\item[(i)] While the semivirial $\mK(T_n^j\phi^j)$  of a Euclidean profile $T_n^j\phi^j$ contains only the information of the partial kinetic energy $\|\nabla_{x}T_n^j\phi^j\|^2_{L^2(\R^3\times\T)}$, in order to apply the Black-Box-Theory on $\R^4$ we need to consider the full virial which contains the complete kinetic energy $\|\nabla_{x,y}\phi^j\|_{L^2(\R^4)}^2$. At the first glance, this seems to be impossible since we are attempting to upgrade some degenerate information into a complete one in the absence of further useful conditions. We shall however see that the missing information concerning the energy $\|\pt_y \phi^j\|_{L^2(\R^4)}^2$ is indeed deeply hidden in the seemingly unrelated condition $\mH(u)<m_{\mM(u)}$, where the proof relies on some highly non-trivial variational estimates.

\item[(ii)] Another issue here is that the Schr\"odinger flow $e^{it_n^j\Delta_{x,y}}$ does not necessarily preserve the Lebesgue norm in the case $t^j_n (N_n^j)^2\to\pm\infty$ (where $t_n^j$ and $N_n^j$ are the time translation and scaling parameters corresponding to the profile $T_n^j\phi^j$). In the defocusing case, this issue can be easily solved by invoking the Sobolev's inequality $\|T_n^j \phi^j\|_{L^4(\R^3\times\T)}\lesssim \|\phi^j\|_{\dot{H}^1(\R^4)}$, as we demand no restriction on the initial data in order to apply the corresponding Black-Box-Theory on $\R^4$. In the focusing case, however, any attempt involving a naive application of the Sobolev's inequality might immediately violate the underlying variational structure and lead to a failure of embedding the Euclidean profile into the Black-Box of the scattering result on $\R^4$. We shall prove that in this case the statement $\|T_n^j \phi^j\|_{L^4(\R^3\times\T)}=o_n(1)$ holds true and leads to the desired claim. Notice when replacing $\R^3\times\T$ to $\R^4$, the claim follows directly from the dispersive estimate on $\R^4$. In our case there will still be an $(N_n^j)^{-\frac14}$-decay missing after applying the dispersive estimate on $\R^3$. We shall prove that the missing decay can be compensated by the torus side by appealing to the Sobolev's inequality on $\T$.
\end{itemize}

At the end of the introductory section, we remark that in the previous works \cite{Luo_Waveguide_MassCritical,Luo_inter} and also in the present paper we have mainly dealt with the case $n=1$. We expect that the framework of semivirial-vanishing geometry should also work in the mass-critical case and improve the results given in \cite{Luo_Waveguide_MassCritical}. Consequently, we expect that the so far developed theory should work equally well in the much harder double critical case $n=2$. These open problems shall provide some interesting topics for future research.

\subsubsection*{Outline of the paper}
The rest of the paper is organized as follows: In Section \ref{sec notation} we list some notation and definitions that will be used throughout the paper. In Section \ref{sec: Dependence} we prove Theorem \ref{thm threshold mass}. Finally, Theorem \ref{main thm} and \ref{thm blow up} are shown in Section \ref{sec: Scattering}.

\subsection{Notation and definitions}\label{sec notation}
We use the notation $A\lesssim B$ whenever there exists some positive constant $C$ such that $A\leq CB$. Similarly we define $A\gtrsim B$ and we use $A\sim B$ when $A\lesssim B\lesssim A$.

For simplicity, we ignore in most cases the dependence of the function spaces on their spatial domains and hide this dependence in their indices. For example $L_x^2=L^2(\R^d)$, $H_{x,y}^1= H^1(\R^d\times \T)$
and so on. However, when the space is involved with time, we still display the underlying temporal interval such as $L_t^pL_x^q(I)$, $L_t^\infty L_{x,y}^2(\R)$ etc. The norm $\|\cdot\|_p$ is defined by $\|\cdot\|_p:=\|\cdot\|_{L_{x,y}^p}$. We shall also consider functions defined on $\R^{d+1}$. In this case, the $\R^{d+1}$-gradient is denoted by $\nabla_{\R^{d+1}}$.

The following quantities will be used throughout the paper: For $u\in H_{x,y}^1$, define
\begin{align}
\mM(u)&:=\|u\|^2_{2},\label{def of mass}\\
\mH(u)&:=\frac{1}{2}\|\nabla_{x,y} u\|^2_{2}-\frac{d-1}{2(d+1)}\|u\|^{2+4/(d-1)}_{2+4/(d-1)},\label{def of mhu}\\
\mK(u)&:=\|\nabla_{x} u\|^2_{2}-\frac{d}{d+1}\|u\|^{2+4/(d-1)}_{2+4/(d-1)},\\
\mI(u)&:=\frac{1}{2}\|\pt_y u\|_{2}^2+\frac{1}{2(d+1)}\|u\|^{2+4/(d-1)}_{2+4/(d-1)}=\mH(u)-\frac{1}{2}
\mK(u).\label{def of mI}
\end{align}
For $\ld\in(0,\infty)$, define
\begin{align}\label{def modified energy}
\mH_{\ld}(u)&:=\frac{\ld}{2}\|\nabla_{y} u\|^2_{2}+\frac{1}{2}\|\nabla_{x} u\|^2_{2}
-\frac{d-1}{2(d+1)}\|u\|^{2+4/(d-1)}_{2+4/(d-1)},\\
\mI_{\ld}(u)&:=\frac{\ld}{2}\|\pt_y u\|_{2}^2+\frac{1}{2(d+1)}\|u\|^{2+4/(d-1)}_{2+4/(d-1)}.\label{def of I ld}
\end{align}
For $u\in H_x^1$, define
\begin{align}
\wmM(u)&:=\|u\|^2_{L_x^2},\label{def of wmM}\\
\wmH(u)&:=\frac{1}{2}\|\nabla_{x} u\|^2_{L_x^2}-\frac{d-1}{2(d+1)}\|u\|^{2+4/(d-1)}_{L_x^{2+4/(d-1)}},\label{def of wmH}\\
\wmI(u)&:=\frac{1}{2(d+1)}\|u\|^{2+4/(d-1)}_{L_x^{2+4/(d-1)}},\label{def of wmI}\\
\wmK(u)&:=\|\nabla_{x} u\|^2_{L_x^2}-\frac{d}{d+1}\|u\|^{2+4/(d-1)}_{L_x^{2+4/(d-1)}}\label{def of wmK}.
\end{align}
For $u\in H^1(\R^{d+1})$, define
\begin{align}
\mH^*(u)&:=\frac{1}{2}\|\nabla_{\R^{d+1}}u\|^2_{L^2(\R^{d+1})}-\frac{d-1}{2(d+1)}\|u\|^{2+4/(d-1)}_{L^{2+4/(d-1)}(\R^{d+1})},\label{def of h star}\\
\mK^*(u)&:=\|\nabla_{\R^{d+1}}u\|^2_{L^2(\R^{d+1})}-\|u\|^{2+4/(d-1)}_{L^{2+4/(d-1)}(\R^{d+1})}.\label{def of k star}
\end{align}
We also define the sets
\begin{align}
S(c)&:=\{u\in H_{x,y}^1:\mM(u)=c\},\\
V(c)&:=\{u\in S(c):\mK(u)=0\},\\
\widehat{S}(c)&:=\{u\in H_x^1:\wmM(u)=c\},\\
\widehat{V}(c)&:=\{u\in \widehat{S}(c):\wmK(u)=0\}
\end{align}
and the variational problems
\begin{align}
m_c&:=\inf\{\mH(u):u\in V(c)\},\label{def of mc}\\
m_{1,\ld}&:=\inf\{\mH_\ld(u):u\in V(1)\},\label{def of auxiliary problem}\\
\wm_c&:=\inf\{\wmH(u):u\in \widehat{V}(c)\}\label{def of wmc}.
\end{align}
Finally, for a function $u\in H_{x,y}^1$, the scaling operator $u\mapsto u^t$ for $t\in(0,\infty)$ is defined by
\begin{align}\label{def of scaling op}
u^t(x,y):=t^{\frac d2}u(tx,y).
\end{align}
The following well-known results concerning the variational problem $\wm_c$ will also be frequently invoked. We refer for instance to \cite{Cazenave2003,Jeanjean1997,Bellazzini2013,BellazziniJeanjean2016} for details of the corresponding proofs.

\begin{lemma}\label{lem wmc property}
The following statements hold true:
\begin{itemize}
\item[(i)]For any $c>0$ the variational problem $\wm_c$ has an optimizer $P_c\in \widehat{S}(c)$. Moreover, $P_c$ satisfies the standing wave equation
\begin{align}\label{standing wave on rd}
-\Delta_x P_c+\omega_c P=|P_c|^{\frac{4}{d-1}} P_c
\end{align}
with some $\omega_c>0$.
\item[(ii)] Any solution $P_c\in H^1(\R^d)$ of \eqref{standing wave on rd} with $\omega_c>0$ is of class $W^{3,p}(\R^d)$ for all $p\in[2,\infty)$.
\item[(iii)] Any solution $P_c\in H^1(\R^d)$ of \eqref{standing wave on rd} satisfies $\wmK(P_c)=0$.
\item[(iv)] The mapping $c\mapsto \wm_c$ is strictly monotone decreasing and continuous on $(0,\infty)$. Moreover, we have $$\lim_{c\to0}\wm_c=\infty\quad\text{and}\quad\lim_{c\to\infty}\wm_c=0.$$
\end{itemize}
\end{lemma}


\section{$y$-dependence of the ground states}\label{sec: Dependence}
In this section we give the proof of Theorem \ref{thm threshold mass}. Similarly as in \cite{TTVproduct2014,Luo_inter} we shall firstly consider the auxiliary problem $m_{1,\ld}$ defined by \eqref{def of auxiliary problem}. Following the same line as in \cite{TTVproduct2014,Luo_inter} we prove the following characterization of $m_{1,\ld}$ for varying $\ld$.

\begin{lemma}\label{lemma auxiliary}
Let $\wm_c$ be the quantity defined through \eqref{def of wmc}. Then there exists some $\ld_*\in(0,\infty)$ such that
\begin{itemize}
\item For all $\ld\in(0,\ld_*)$ we have $m_{1,\ld}<2\pi \wm_{(2\pi)^{-1}}$.
\item For all $\ld\in(\ld_*,\infty)$ we have $m_{1,\ld}=2\pi \wm_{(2\pi)^{-1}}$. Moreover, any minimizer $u_\ld$ of $m_{1,\ld}$ must satisfy $\pt_y u_\ld=0$.
\end{itemize}
\end{lemma}
The proof of Theorem \ref{thm threshold mass} follows then from Lemma \ref{lemma auxiliary} by simple rescaling arguments. We underline that in comparison to the models studied in \cite{TTVproduct2014,Luo_inter}, the new challenge here is to identify a non-vanishing weak limit of a minimizing sequence for the variational problem $m_c$ (respectively $m_{1,\ld}$), where the concentration compactness arguments in \cite{TTVproduct2014} fail in the energy-critical setting. The key observation is that for large mass $c$ (respectively large $\ld$) we are able to prove that the $L_{x,y}^{2+4/(d-1)}$-norm of a minimizing sequence $(u_n)_n$ will concentrate to the zero Fourier coefficient $m(u_n):=(2\pi)^{-1}\int_{\T}u_n(x,y)\,dy$ of $u_n$ w.r.t. the $y$-direction. In this case we may appeal to the classical concentration compactness arguments on $\R^d$ to identify a non-vanishing weak limit.

\subsection{Some auxiliary preliminaries}\label{sec: Existence}
As a starting point, we collect in this subsection some usefully auxiliary results from \cite{Luo_inter}. We shall simply omit the proofs and refer to \cite{Luo_inter} for further details.

\begin{lemma}[Scale-invariant Gagliardo-Nirenberg inequality on $\R^d\times\T$]\label{lemma gn additive}
There exists some $C>0$ such that for all $u\in H_{x,y}^1$ we have
\begin{align*}
\|u\|_{\frac{2(d+1)}{d-1}}^{\frac{2(d+1)}{d-1}}\leq C\|\nabla_x u\|_2^{\frac{2d}{d-1}}
(\| u\|_{2}^{\frac{2}{d-1}}+\|\pt_y u\|_{2}^{\frac{2}{d-1}}).
\end{align*}
\end{lemma}

\begin{lemma}[Lower and upper bound of $m_c$]\label{cor lower bound 1}
For any $c\in(0,\infty)$ we have $m_c\in(0,\infty)$, where $m_c$ is defined by \eqref{def of mc}.
\end{lemma}

\begin{lemma}[Property of the mapping $t\mapsto \mK(u^t)$]\label{monotoneproperty}
Let $c>0$ and $u\in S(c)$. Then the following statements hold true:
\begin{enumerate}
\item[(i)] $\frac{\partial}{\partial t}\mH(u^t)=t^{-1} Q(u^t)$ for all $t>0$.
\item[(ii)] There exists some $t^*=t^*(u)>0$ such that $u^{t^*}\in V(c)$.
\item[(iii)] We have $t^*<1$ if and only if $\mK(u)<0$. Moreover, $t^*=1$ if and only if $\mK(u)=0$.
\item[(iv)] Following inequalities hold:
\begin{equation*}
Q(u^t) \left\{
\begin{array}{lr}
             >0, &t\in(0,t^*) ,\\
             <0, &t\in(t^*,\infty).
             \end{array}
\right.
\end{equation*}
\item[(v)] $\mH(u^t)<\mH(u^{t^*})$ for all $t>0$ with $t\neq t^*$.
\end{enumerate}
\end{lemma}

\begin{lemma}[Property of the mapping $c\mapsto m_c$]\label{monotone lemma}
The mapping $c\mapsto m_c$ is continuous and monotone decreasing on $(0,\infty)$.
\end{lemma}

\begin{lemma}[Characterization of a minimizer as a standing wave solution]\label{minimizer is solution}
For any $c\in(0,\infty)$ an optimizer $u$ of $m_c$ is a solution of
\begin{align}\label{standing wave}
-\Delta_{x,y}u+\beta u=|u|^{\frac{4}{d-1}}u
\end{align}
with some $\beta\in\R$.
\end{lemma}

\subsection{Existence of ground states with large mass}
In order to initiate the proof of Lemma \ref{lemma auxiliary} we will need to prove that the variational problem $m_{1,\ld}$ has an optimizer $u_{\ld}$ for all sufficiently large $\ld$. By rescaling, this is equivalent to show that the variational problem $m_c$ has an optimizer $u_c$ for all sufficiently large $c$. This statement will be given as Proposition \ref{thm existence of ground state 1} below. We also point out that Proposition \ref{thm existence of ground state 1} is of independent interest in the sense that its proof is indeed available for any mass as long as the underlying minimizing sequence has a non-vanishing weak limit (which at the moment is only known to be true for large mass).

\begin{proposition}[Existence of ground states with large mass]\label{thm existence of ground state 1}
There exist $c^*\in[0,\infty)$ such that for any $c\in(c^*,\infty)$ the minimization problem $m_c$ has a positive optimizer $u_c$. Moreover, $u_c$ solves the standing wave equation \eqref{standing wave} with some $\beta=\beta_c>0$.
\end{proposition}

\begin{proof}
We split our proof into five steps.
\subsubsection*{Step 1: Non-vanishing weak limit of a minimizing sequence}
As a starting point, we firstly show that for all sufficiently large mass, a minimizing sequence $(u_n)_n\subset V(c)$ of $m_c$ shall always weakly converge (up to a subsequence and $\R^d_x$-translations) to a non-vanishing function $u$ in $H_{x,y}^1$. We start with showing that $(u_n)_n$ is a bounded sequence in $H_{x,y}^1$. Indeed, using Lemma \ref{cor lower bound 1} and the fact that $\mK(u_n)=0$ we infer that for all sufficiently large $n$
\begin{align}\label{gn ineq 5}
\infty>2m_c\geq\mH(u_n)=\mH(u_n)-\frac{d-1}{2d}\mK(u_n)=\frac{1}{2}\|\pt_y u_n\|_2^2+\frac{1}{2d}\|\nabla_x u_n\|_2^2,
\end{align}
which in turn implies the $H_{x,y}^1$-boundedness of $(u_n)_n$. Next define $m(u):=(2\pi)^{-1}\int_{\T}u(y)\,dy$. Using triangular inequality and $\mK(u_n)=0$ we obtain
\begin{align}\label{gn ineq 4}
\|m(u_n)\|_{\frac{2(d+1)}{d-1}}&\geq \|u_n\|_{\frac{2(d+1)}{d-1}}-\|u_n-m(u_n)\|_{\frac{2(d+1)}{d-1}}\nonumber\\
&= \bg(\frac{d+1}{d}\bg)^{\frac{d-1}{2(d+1)}}\|\nabla_x u_n\|_2^{\frac{d-1}{d+1}}-\|u_n-m(u_n)\|_{\frac{2(d+1)}{d-1}}.
\end{align}
To handle the term $\|u_n-m(u_n)\|_{\frac{2(d+1)}{d-1}}$, we firstly recall the following well-known Sobolev's inequality on $\T$ for functions with zero mean (see for instance \cite{sobolev_torus}):
\begin{align}\label{sobolev torus 1}
\|u-m(u)\|_{L_y^{\frac{2(d+1)}{d-1}}}\lesssim\|u\|_{\dot{H}_y^{\frac{1}{d+1}}}.
\end{align}
Writing $u$ into the Fourier series $u(x,y)=\sum_{k} e^{iky}u_k(x)$ w.r.t $y$ and followed by \eqref{sobolev torus 1}, Minkowski, Gagliardo-Nirenberg on $\R^d$ and H\"older we obtain
\begin{align}\label{gn ineq 3}
\|u-m(u)\|_{\frac{2(d+1)}{d-1}}&\lesssim \|(|k|^{\frac{1}{d+1}}|u_k|)\|_{L_x^{2(d+1)/(d-1)}\ell_k^2}
\leq \|(|k|^{\frac{1}{d+1}}|u_k|)\|_{\ell_k^2 L_x^{2(d+1)/(d-1)}}\nonumber\\
&\lesssim \|((|k|\|u_k\|_{L_x^2})^{\frac{1}{d+1}}\|\nabla_x u_k\|_{L_x^2}^{\frac{d}{d+1}})_k\|_{\ell_k^2}\nonumber\\
&\lesssim \|(k\|u_k\|_{L_x^2})_k\|^{\frac{1}{d+1}}_{\ell_k^2}\|(\|\nabla_x u_k\|_{L_x^2})_k\|^{\frac{d}{d+1}}_{\ell_k^2}
=\|\pt_y u\|_2^{\frac{1}{d+1}}\|\nabla_x u\|_2^{\frac{d}{d+1}}.
\end{align}
Thus \eqref{gn ineq 5}, \eqref{gn ineq 4} and \eqref{gn ineq 3} imply that there exist some positive constants $C=C(d)>0$ such that for all sufficiently large $n$
\begin{align}\label{addition estimate}
\|m(u_n)\|_{\frac{2(d+1)}{d-1}}\geq C\|\nabla_x u_n\|_2^{\frac{d-1}{d+1}}(1-\|\pt_y u\|_2^{\frac{1}{d+1}}\|\nabla_x u_n\|_2^{\frac{1}{d+1}}).
\end{align}
By assuming that a function $u\in V(c)$ is independent of $y$ we infer that $m_c\leq 2\pi \wm_{(2\pi)^{-1}c}$. Combining with Lemma \ref{lem wmc property} (iv) we deduce $\lim_{c\to\infty}m_c=0$. Thus using \eqref{gn ineq 5} we know that for all sufficiently large $c$ there exists some sufficiently large $N=N(c)\in\N$ such that $\|\pt_y u_n\|_2^{\frac{1}{d+1}}\|\nabla_x u_n\|_2^{\frac{1}{d+1}}\leq \frac{1}{2}$ for all $n\geq N$. On the other hand, by Lemma \ref{lemma gn additive} and $\mK(u_n)=0$ we obtain
\begin{align}\label{addition estimate 2}
\|\nabla_x u_n\|_2^2=\frac{d}{d+1} \|u_n\|_{2+4/(d-1)}^{2+4/(d-1)}\lesssim \|\nabla_x u_n\|_2^{\frac{2d}{d-1}},
\end{align}
and we conclude that $\liminf_{n\to\infty}\|\nabla_x u_n\|_2>0$. Summing up we infer that
\begin{align*}
\liminf_{n\to\infty}\|m(u_n)\|_{\frac{2(d+1)}{d-1}}\gtrsim \liminf_{n\to\infty}\|\nabla_x u_n\|_{2}^{\frac{d-1}{d+1}}>0
\end{align*}
for all sufficiently large $c$. Since $(u_n)_n$ is a bounded sequence in $H_{x,y}^1$, we know that $(m(u_n))_n$ is a bounded sequence in $H_x^1$. Notice also that the exponent $\frac{2(d+1)}{d-1}$ lies in the intercritical regime $\in(2,2+\frac{4}{d-2})$, thus by the classical concentration compactness arguments on $\R^d$ (see for instance \cite{Lions1984I}) we can find some $(x_n)_n\subset\R^d$ and $v\in H_x^1\setminus\{0\}$ such that
\begin{align*}
m(u_n)(x+x_n)\rightharpoonup v(x)\quad\text{weakly in $H_x^1$}.
\end{align*}
On the other hand, the sequence $(u_n(x+x_n,y))_n$ is also a bounded $H_{x,y}^1$-minimizing sequence of $m_c$. If we denote the weak $H_{x,y}^1$-limit (up to a subsequence) of $(u_n(x+x_n,y))_n$ by $u$, then $m(u)=v$. In particular we infer that $u\neq 0$, which in turn completes the proof of Step 1.
\subsubsection*{Step 2: A Le Coz characterization of $m_c$}
Next, we shall give a different and much handier characterization for $m_c$ due to Le Coz \cite{LeCoz2008} that is more useful for our analysis. Define
\begin{align}
\tilde{m}_c:=\inf\{\mI(u):u\in S(c),\mK(u)\leq 0\},\label{mtilde equal m}
\end{align}
where $\mI(u)$ is the energy functional defined by \eqref{def of mI}. We aim to prove $m_c=\tilde{m}_c$. Let $(u_n)_n\subset S(c)$ be a minimizing sequence for the variational problem $\tilde{m}_c$, i.e.
\begin{align}\label{le coz charac}
\mI(u_n)=\tilde{m}_c+o_n(1),\quad
\mK(u_n)\leq 0\quad\forall\,n\in\N.
\end{align}
By Lemma \ref{monotoneproperty} we know that there exists some $t_n\in(0,1]$ such that $\mK(u_n^{t_n})$ is equal to zero. Thus
\begin{align*}
m_c\leq \mH(u_n^{t_n})=\mI(u_n^{t_n})\leq \mI(u_n)=\tilde{m}_c+o_n(1).
\end{align*}
Sending $n\to\infty$ we infer that $m_c\leq \tilde{m}_c$. On the other hand,
\begin{align*}
\tilde{m}_c
\leq\inf\{\mI(u):u\in V(c)\}
=\inf\{\mH(u):u\in V(c)\}=m_c,
\end{align*}
which completes the proof.

\subsubsection*{Step 3: Existence of a non-negative optimizer of $m_c$}
Define
\begin{align*}
c^*:=\inf\{\hat{c}\in(0,\infty):\text{ $m_c$ has a minimizing sequence with non-vanishing $H_{x,y}^1$-weak limit $\forall\,c\geq\hat{c}$}\}.
\end{align*}
By Step 1 we know that $c^*\in[0,\infty)$ and from now on we shall fix some $c\in(c^*,\infty)$. Let $(u_n)_n\subset V(c)$ be a minimizing sequence of $m_c$ which also possesses a non-vanishing $H_{x,y}^1$-weak limit $u\neq 0$. Using diamagnetic inequality and Step 2, by replacing $u_n$ and $u$ to $|u_n|$ and $|u|$ respectively we may assume that all $u_n$ and $u$ are non-negative, $\mK(u_n)\leq 0$ and $\mI(u_n)$ approaches $\tilde{m}_{c}$. By weakly lower semicontinuity of norms we deduce
\begin{align}
\mM(u)=:c_1\in(0,c],\quad\mI(u)\leq \tilde{m}_c.\label{xia jie}
\end{align}
We next show $\mK(u)\leq 0$. Assume the contrary $\mK(u)>0$. By Brezis-Lieb lemma, $\mK(u_n)\leq 0$ and the fact that $L_{x,y}^2$ is a Hilbert space we infer that
\begin{align*}
\mM(u_n-u)&=c-c_1+o_n(1),\\
\mK(u_n-u)&\leq -\mK(u)+o_n(1).
\end{align*}
Therefore, for all sufficiently large $n$ we know that $\mM(u_n-u)\in(0,c)$ and $\mK(u_n-u)<0$. By Lemma \ref{monotoneproperty} we also know that there exists some $t_n\in(0,1)$ such that $\mK((u_n-u)^{t_n})=0$. Consequently, Lemma \ref{monotone lemma}, Brezis-Lieb lemma and Step 2 yield
\begin{align*}
\tilde{m}_c\leq \mI((u_n-u)^{t_n})<\mI(u_n-u)=\mI(u_n)-\mI(u)+o_n(1)=\tilde{m}_c-\mI(u)+o_n(1).
\end{align*}
Sending $n\to\infty$ and using the non-negativity of $\mI(u)$ we obtain $\mI(u)=0$. This in turn implies $u=0$, which is a contradiction and thus $\mK(u)\leq 0$. If $\mK(u)<0$, then again by Lemma \ref{monotoneproperty} we find some $s\in(0,1)$ such that $\mK(u^s)=0$. But then using Lemma \ref{monotone lemma}, Step 2 and the fact $c_1\leq c$
\begin{align*}
\tilde{m}_{c_1}\leq \mI(u^s)<\mI(u)\leq \tilde{m}_c\leq \tilde{m}_{c_1},
\end{align*}
a contradiction. We conclude therefore $\mK(u)=0$

Thus $u$ is a minimizer of $m_{c_1}$. From Lemma \ref{minimizer is solution} we know that $u$ is a solution of \eqref{standing wave} and it remains to show that the corresponding $\beta$ in \eqref{standing wave} is positive and $\mM(u)=c$.

\subsubsection*{Step 4: Positivity of $\beta$}
First we prove that $\beta$ is non-negative. Testing \eqref{standing wave} with $u$ and followed by eliminating $\|\nabla_x u\|_2^2$ using $\mK(u)=0$ we obtain
\begin{align}
\|\pt_y u\|_2^2+\beta\mM(u)=\frac{1}{d+1}\|u\|_{2(d+1)/(d-1)}^{2(d+1)/(d-1)}.\label{corrected1}
\end{align}
Next, we define the scaling operator $T_\ld$ by
\begin{align}\label{def of t ld}
T_\ld u(x,y):=\ld^{\frac{d-1}{2}}u(\ld x,y).
\end{align}
Then
\begin{align*}
\|T_\ld(\nabla_x u)\|_2^2&=\ld\|\nabla_x u\|_2^2,\\
\|T_\ld u\|_{2(d+1)/(d-1)}^{2(d+1)/(d-1)}&=\ld\|u\|_{2(d+1)/(d-1)}^{2(d+1)/(d-1)},\\
\mK(T_\ld u)&=\ld\mK(u),\\
\|T_\ld (\pt_y u)\|_2^2&=\ld^{-1}\|\pt_y u\|_2^2,\\
\|T_\ld u\|_2^2&=\ld^{-1}\|u\|_2^2.
\end{align*}
Using Lemma \ref{monotone lemma} and the fact that $u_c$ is an optimizer of $m_c$ we infer that $\frac{d}{d\ld}\mH(T_\ld u_c)|_{\ld=1}\geq 0$, or  equivalently
\begin{align}\label{identity abcd}
\|\pt_y u\|_{2}^2\leq \frac{1}{d+1}\|u\|_{2(d+1)/(d-1)}^{2(d+1)/(d-1)}.
\end{align}
Combining with \eqref{corrected1} we deduce $\beta\mM(u)\geq 0$. Since $u\neq 0$, we conclude that $\beta\geq 0$. It is left to show that $\beta=0$ leads to a contradiction, which completes the proof of Step 4. Assume therefore that $u$ satisfies the equation
\begin{align}\label{liouville}
-\Delta_{x,y}u=u^{\frac{d+3}{d-1}}.
\end{align}
By the Brezis-Kato estimate \cite{BrezisKato} (see also \cite[Lem. B.3]{Struwe1996}) and the local $L^p$-elliptic regularity (see for instance \cite[Lem. B.2]{Struwe1996}) we know that $u\in W^{2,p}_{\rm loc}(\R^{d+1})$ for all $p\in[1,\infty)$. Hence by Sobolev embedding we also know that $u$ and $\nabla u$ are of class $L^\infty_{\rm loc}(\R^{d+1})$. Taking $\pt_{j}$ to \eqref{liouville} with $j\in\{1,\cdots,d+1\}$ we obtain
\begin{align*}
-\Delta_{x,y}\pt_j u=\frac{d+3}{d-1}(u^{\frac{4}{d-1}}\pt_j u)\in L^\infty_{\rm loc}(\R^{d+1}).
\end{align*}
Hence by applying the local $L^p$-elliptic regularity again we deduce $u\in W^{3,p}_{\rm loc}(\R^{d+1})$ for all $p\in[1,\infty)$. Consequently, by Sobolev embedding we infer that $u\in C^2(\R^{d+1})$. Using strong maximum principle we also know that $u$ is positive. By \cite{Liouville3}, any positive $C^2$-solution of \eqref{liouville} must be of the form
$$ u(x,y)=b\bg(\frac{a}{1+a^2|(x,y)-(x_0,y_0)|^2}\bg)^{\frac{d-1}{2}}$$
with some $a,b>0$ and $(x_0,y_0)\in\R^{d+1}$. However, in this case $u$ can not be periodic along the $y$-direction, which leads to a contradiction. This completes the proof of Step 4.

\subsubsection*{Step 5: $\mM(u)=c$ and conclusion}
Finally, we prove $\mM(u)=c$. Assume therefore $c_1<c$. By Lemma \ref{monotone lemma} and \eqref{xia jie} we know that $m_{c_1}$ is a local minimizer of the mapping $c\mapsto m_c$, which in turn implies that the inequality in \eqref{identity abcd} is in fact an equality. Now using \eqref{identity abcd} (as an equality) and \eqref{corrected1} we infer that $\beta\mM(u)=0$, which is a contradiction since $\beta>0$ and $u\neq 0$. We thus conclude $\mM(u)=c$. That $u$ is positive follows immediately from the strong maximum principle. This completes the desired proof.
\end{proof}

\subsection{Proof of Lemma \ref{lemma auxiliary}}
Before we finally give the proof of Lemma \ref{lemma auxiliary}, we still need state some usefully auxiliary lemmas.

\begin{lemma}\label{additional exitence lemma}
For all sufficiently large $\ld$ the minimization problem $m_{1,\ld}$ has a positive minimizer.
\end{lemma}

\begin{proof}
It is easy to check that $u\mapsto T_{c} u$ (where $T_{c} u$ is given by \eqref{def of t ld}) defines a bijection between $V(c)$ and $V(1)$. Direct calculation also shows that $\mH(u)=c^{-1}\mH_{c^2}(T_{c}u)$, thus $m_{c}=c^{-1}m_{1,c^2}$. The desired claim follows from the existence claim of positive optimizers of $m_c$ for large $c$ deduced in Proposition \ref{thm existence of ground state 1}.
\end{proof}

\begin{lemma}\label{auxiliary lemma 1}
We have
\begin{align}\label{limit ld to infty energy sec4}
\lim_{\ld\to\infty}m_{1,\ld}=2\pi\wm_{(2\pi)^{-1}}.
\end{align}
Additionally, for all sufficiently large $\ld$ let $u_\ld\in V(1)$ be a positive optimizer of $m_{1,\ld}$ which also satisfies
\begin{align}
-\Delta_x u_\ld-\ld \pt_y^2 u_\ld+\beta_\ld u_\ld=|u_\ld|^{\frac{4}{d-1}} u_\ld\quad\text{on $\R^d\times \T$}\label{vanishing 3 sec4}
\end{align}
for some $\beta_\ld>0$. Then
\begin{align}
\lim_{\ld\to\infty}\ld\|\pt_y u_\ld\|_2^2=0.\label{vanishing sec4}
\end{align}
\end{lemma}

\begin{remark}
That the frequency exponent $\beta_\ld$ is positive follows from Lemma \ref{additional exitence lemma} and Step 4 of the proof of Proposition \ref{thm existence of ground state 1}
\end{remark}

\begin{proof}
By assuming that a candidate in $V(1)$ is independent of $y$ we already conclude
\begin{align}
m_{1,\ld}\leq 2\pi \wm_{(2\pi)^{-1}}.\label{upper bound sec4}
\end{align} Next we prove
\begin{align}
\lim_{\ld\to\infty}\|\pt_y u_\ld\|_2^2=0.\label{vanishing1 sec4}
\end{align}
Suppose that \eqref{vanishing1 sec4} does not hold. Then we must have
\begin{align*}
\lim_{\ld\to\infty}\ld\|\pt_y u_\ld\|_2^2=\infty.
\end{align*}
Since $\mK(u_\ld)=0$,
\begin{align}
m_{1,\ld}=\mH_\ld(u_\ld)-\frac{d-1}{2 d}\mK(u_\ld)
=\frac{\ld}{2}\|\pt_y u_\ld\|_2^2
+\frac{1}{2d}\|\nabla_x u_\ld\|_2^2
\geq \frac{\ld}{2}\|\pt_y u_\ld\|_2^2\to\infty\label{contradiction sec4}
\end{align}
as $\ld\to\infty$, which contradicts \eqref{upper bound sec4} and in turn proves \eqref{vanishing1 sec4}. Using \eqref{upper bound sec4} and \eqref{contradiction sec4} we infer that
\begin{align}\label{upper bound 2 sec4}
\|\nabla_x u_\ld\|_2^2\lesssim  m_{1,\ld}\leq  2\pi\wm_{(2\pi)^{-1}}<\infty.
\end{align}
Therefore $(u_\ld)_\ld$ is a bounded sequence in $H_{x,y}^1$, whose weak limit is denoted by $u$. Since $\|\pt_y u_\ld\|_2^2\to 0$, using \eqref{addition estimate} and \eqref{addition estimate 2} we conclude that
\begin{align*}
\liminf_{\ld\to\infty}\|m(u_\ld)\|_{2+4/(d-1)}^{2+4/(d-1)}>0.
\end{align*}
Thus arguing as in Step 1 of Proposition \ref{thm existence of ground state 1} we may also assume that $u\neq 0$. By \eqref{vanishing1 sec4} we know that $u$ is independent of $y$ and thus $u\in H_x^1$. Moreover, using weakly lower semicontinuity of norms we know that $\wmM(u)\in(0,(2\pi)^{-1}]$. On the other hand, using $\mK(u_\ld)=0$, \eqref{vanishing 3 sec4} and $\mM(u_\ld)=1$ we obtain
\begin{align*}
\beta_\ld=\frac{1}{d+1}\|u_\ld\|_{2(d+1)/(d-1)}^{2(d+1)/(d-1)}-\ld\|\pt_y u_\ld\|_2^2\lesssim \|u_\ld\|_{2(d+1)/(d-1)}^{2(d+1)/(d-1)}.
\end{align*}
Thus $(\beta_\ld)_\ld$ is a bounded sequence in $(0,\infty)$, whose limit is denoted by $\beta$. We now test \eqref{vanishing 3 sec4} with $\phi\in C_c^\infty(\R^d)$ and integrate both sides over $\R^d\times\T$. Notice particularly that the term $\int_{\R^d\times\T}  \pt_y^2 u_\ld \phi\,dxdy=0$ for any $\ld>0$ since $\phi$ is independent of $y$. Using the weak convergence of $u_\ld$ to $u$ in $H_{x,y}^1$, by sending $\ld\to\infty$ we obtain
\begin{align}
-\Delta_x u+\beta u=|u|^{\frac{4}{d-1}} u\quad\text{in $\R^d$}.\label{vanishing 4 sec4}
\end{align}
In particular, by Lemma \ref{lem wmc property} we know that $\wmK(u)=0$ and consequently $\beta>0$. Combining with weakly lower semicontinuity of norms we deduce
$$2\pi\wmH(u)=2\pi\wmI(u)\leq\liminf_{\ld\to\infty}\mI_\ld(u_\ld)= \liminf_{\ld\to\infty}\mH_\ld(u_\ld)\leq 2\pi \wm_{(2\pi)^{-1}},$$
where $\wmH(u)$, $\wmI(u)$, $\mI_\ld(u)$ and $\mH_\ld(u)$ are the quantities defined by \eqref{def of wmH}, \eqref{def of wmI}, \eqref{def of I ld} and \eqref{def modified energy} respectively. However, by Lemma \ref{lem wmc property} the mapping $c\mapsto \wm_c$ is strictly monotone decreasing on $(0,\infty)$, from which we conclude that $\wmM(u)=(2\pi)^{-1}$ and $u$ is an optimizer of $\wm_{(2\pi)^{-1}}$. Using the weakly lower semicontinuity of norms we obtain
\begin{align}
m_{1,\ld}&=\mH_\ld(u_\ld)=\mH_\ld(u_\ld)-\frac{d-1}{2 d}\mK(u)
=\frac{\ld}{2}\|\pt_y u_\ld\|_2^2+\frac{1}{2d}\|\nabla_x u_\ld\|_2^2\nonumber\\
&\geq \frac{1}{2d}\|\nabla_x u_\ld\|_2^2
\geq \frac{2\pi}{2d}\|\nabla_x u\|_{L_x^2}^2+o_\ld(1)
=2\pi \wmH(u)+o_\ld(1)\geq 2\pi\wm_{(2\pi)^{-1}}+o_\ld(1).\label{vanishing 5 sec4}
\end{align}
Letting $\ld\to\infty$ and taking \eqref{upper bound sec4} into account yield \eqref{limit ld to infty energy sec4}. Finally, \eqref{vanishing sec4} follows directly from the previous calculation by not neglecting $\ld\|u_\ld\|_2^2$ therein. This completes the desired proof.
\end{proof}

\begin{lemma}\label{strong convergence u ld}
Let $u_\ld$ and $u$ be the functions given in the proof of Lemma \ref{auxiliary lemma 1}. Then $u_\ld\to u$ strongly in $H_{x,y}^1$.
\end{lemma}

\begin{proof}
In fact, in the proof of Lemma \ref{auxiliary lemma 1} we see that all the inequalities involving the weakly lower semicontinuity of norms are in fact equalities. This particularly implies $\|u_\ld\|_{H_{x,y}^1}\to \|u\|_{H_{x,y}^1}=2\pi\|u\|_{H_x^1}$ as $\ld\to\infty$, which in turn implies the strong convergence of $u_\ld$ to $u$ in $H_{x,y}^1$.
\end{proof}

\begin{lemma}\label{lemma no dependence}
There exists some $\ld_0$ such that $\pt_y u_\ld=0$ for all $\ld>\ld_0$.
\end{lemma}

\begin{proof}
Let $w_\ld:=\pt_y u_\ld$. Then taking $\pt_y$-derivative to \eqref{vanishing 3 sec4} we obtain
\begin{align}\label{no dependence 1}
-\Delta_x w_\ld-\ld \pt^2_y w_\ld+\beta_\ld w_\ld=\pt_y(|u_\ld|^{\frac{4}{d-1}} u_\ld)=\frac{d+3}{d-1}|u_\ld|^{\frac{4}{d-1}} w_\ld.
\end{align}
Testing \eqref{no dependence 1} with $w_\ld$ and rewriting suitably, we infer that
\begin{align}
0&=\|\nabla_x w_\ld\|_2^2+\ld\|\pt_y w_\ld\|_2^2+\beta_\ld \|w_\ld\|_2^2-{\frac{d+3}{d-1}}\int_{\R^d\times \T}|u_\ld|^{\frac{4}{d-1}} |\bar{w}_\ld|^2\,dxdy\nonumber\\
&=(\ld-1)\|\pt_y w_\ld\|_2^2-{\frac{d+3}{d-1}}\int_{\R^d\times \T}|u|^{\frac{4}{d-1}} |w_\ld|^2\,dxdy\label{no dependence 2}\\
&+\beta_\ld \|w_\ld\|_2^2+\|\nabla_{x,y}w_\ld\|_2^2\label{no dependence 4}\\
&-{\frac{d+3}{d-1}}\int_{\R^d\times \T}(|u_\ld|^{\frac{4}{d-1}} -|u|^{\frac{4}{d-1}})|w_\ld|^2\,dxdy.\label{no dependence 3}
\end{align}
For \eqref{no dependence 2}, we firstly point out that by Lemma \ref{lem wmc property} (ii) and Sobolev embedding we have $u\in L^\infty(\R^d)$. On the other hand, since $\int_\T w_\ld\,dy=0$, we have $\|w_\ld\|_2\leq \|\pt_y w_\ld\|_2$. Summing up, we conclude that
\begin{align*}
\eqref{no dependence 2}\geq (\ld-1-{\frac{d+3}{d-1}}\|u\|^{\frac{4}{d-1}}_{L_x^\infty})\|\pt_y w_\ld\|_2^2\geq 0
\end{align*}
for all sufficiently large $\ld$. For \eqref{no dependence 3}, we discuss the cases ${\frac{4}{d-1}}\leq 1$ and ${\frac{4}{d-1}}>1$ separately. For ${\frac{4}{d-1}}\leq 1$, we estimate the second term in \eqref{no dependence 3} using subadditivity of concave function, H\"older's inequality, Lemma \ref{strong convergence u ld} and the Sobolev embedding $H_{x,y}^1\hookrightarrow L_{x,y}^{2(d+1)/(d-1)}$:
\begin{align*}
&\,\int_{\R^d\times \T}(|u_\ld|^{\frac{4}{d-1}} -|u|^{\frac{4}{d-1}})|w_\ld|^2\,dxdy\nonumber\\
\leq&\,\int_{\R^d\times \T}|u_\ld-u|^{\frac{4}{d-1}}|w_\ld|^2\,dxdy\nonumber\\
\leq& \,\|u_\ld-u\|_{2+4/(d-1)}^{\frac{4}{d-1}}\|w_\ld\|_{2+4/(d-1)}^2\leq o_\ld(1)\|w_\ld\|^2_{H_{x,y}^1}.
\end{align*}
The case ${\frac{4}{d-1}}>1$ can be similarly estimated as follows:
\begin{align*}
&\,\int_{\R^d\times \T}(|u_\ld|^{\frac{4}{d-1}} -|u|^{\frac{4}{d-1}})|w_\ld|^2\,dxdy\nonumber\\
\lesssim&\,\int_{\R^d\times \T}|u_\ld-u||w_\ld|^2(|u_\ld|^{{\frac{4}{d-1}}-1}+|u|^{{\frac{4}{d-1}}-1})\,dxdy\nonumber\\
\leq& \,\|u_\ld-u\|_{2+4/(d-1)}(\|u_\ld\|^{{\frac{4}{d-1}}-1}_{2+4/(d-1)}+\|u\|^{{\frac{4}{d-1}}-1}_{2+4/(d-1)})\|w_\ld\|_{2+4/(d-1)}^2\leq o_\ld(1)\|w_\ld\|^2_{H_{x,y}^1}.
\end{align*}
Therefore, \eqref{no dependence 2}, \eqref{no dependence 3}, the facts $\beta_\ld=\beta+o_\ld(1)$ and $\beta>0$ then imply
\begin{align*}
0\gtrsim \|w_\ld\|_{H_{x,y}^1}^2(1-o_\ld(1))\gtrsim \|w_\ld\|_{H_{x,y}^1}^2
\end{align*}
for all $\ld>\ld_0$ with some sufficiently large $\ld_0$. We therefore conclude that $0=w_\ld=\pt_y u_\ld$ for all $\ld>\ld_0$.
\end{proof}

Having all the preliminaries we are in a position to prove Lemma \ref{lemma auxiliary}.

\begin{proof}[Proof of Lemma \ref{lemma auxiliary}]
Define
\begin{align*}
\ld_*:=\inf\{\tilde{\ld}\in(0,\infty):m_{1,\ld}=2\pi \wm_{(2\pi)^{-1}}\,\forall\,\ld\geq \tilde{\ld}\}.
\end{align*}
By Lemma \ref{lemma no dependence} we know that $\ld_*<\infty$. Next we show $\ld_*>0$. It suffices to show
\begin{align}\label{strictly small}
\lim_{\ld\to 0}m_{1,\ld}<2\pi\wm_{(2\pi)^{-1}}.
\end{align}
To see this, we firstly define the function $\rho:[0,2\pi]\to[0,\infty)$ as follows: Let $a\in(0,\pi)$ such that $a>\pi-3\pi\bg(\frac{3}{3+4/(d-1)}\bg)^{\frac{d-1}{2}}$. This is always possible for $a$ sufficiently close to $\pi$. Then we define $\rho$ by
\begin{align*}
\rho(y)=\left\{
\begin{array}{ll}
0,&y\in[0,a]\cup[2\pi-a,2\pi],\\
(\pi-a)^{-1}\bg(\frac{3+4/(d-1)}{3}\bg)^{\frac{d-1}{4}}(y-a),&y\in[a,\pi],\\
(\pi-a)^{-1}\bg(\frac{3+4/(d-1)}{3}\bg)^{\frac{d-1}{4}}(2\pi-a-y),&y\in[\pi,2\pi-a].
\end{array}
\right.
\end{align*}
By direct calculation we see that $\rho\in H_y^1$ and
\begin{align*}
2\pi>\|\rho\|_{L_y^2}^2=\|\rho\|_{L_y^{2+4/(d-1)}}^{2+4/(d-1)}.
\end{align*}
Next, let $P\in H_x^1$ be an optimizer of $\wm_{\|\rho\|_{L_y^2}^{-2}}$. Since by Lemma \ref{lem wmc property} the mapping $c\mapsto\wm_c$ is strictly decreasing on $(0,\infty)$ and $\|\rho\|_{L_y^2}^{-2}>(2\pi)^{-1}$, we infer that $\wm_{\|\rho\|_{L_y^2}^{-2}}<\wm_{(2\pi)^{-1}}$. Furthermore, by Lemma \ref{lem wmc property} (iii) we have $\|\nabla_x P\|_{L_x^2}^2=\frac{d}{d+1}\|P\|_{L_x^{2+4/(d-1)}}^{2+4/(d-1)}$. Now define $\psi(x,y):=\rho(y)P(x)$. Then we conclude that $\psi\in H_{x,y}^1$, $\mM(\psi)=\|\rho\|_{L_y^2}^2\wmM(P)=1$,
\begin{align*}
\mK(\psi)&=\|\nabla_x\psi\|_2^2-\frac{d}{d+1}\|\psi\|_{2+4/(d-1)}^{2+4/(d-1)}\nonumber\\
&=\|\rho\|_{L_y^2}^2\|\nabla_x P\|_{L_x^2}^2
-\frac{ d}{d+1}\|\rho\|_{L_y^{2+4/(d-1)}}^{2+4/(d-1)}\|P\|_{L_x^{2+4/(d-1)}}^{2+4/(d-1)}\nonumber\\
&=\|\rho\|_{L_y^2}^2\bg(\|\nabla_x P\|_{L_x^2}^2-\frac{ d}{d+1}\|P\|_{L_x^{2+4/(d-1)}}^{2+4/(d-1)}\bg)=0
\end{align*}
and
\begin{align*}
\mH_{*}(\psi)
&:=\frac{1}{2}\|\nabla_x \psi\|_{2}^2-\frac{1}{2+4/(d-1)}\|\psi\|_{{2+4/(d-1)}}^{2+4/(d-1)}\nonumber\\
&=\frac{1}{2}\|\rho\|_{L_y^2}^2\|\nabla_x P\|_{L_x^2}^2
-\frac{1}{2+4/(d-1)}\|\rho\|_{L_y^{2+4/(d-1)}}^{2+4/(d-1)}\|P\|_{L_x^{2+4/(d-1)}}^{2+4/(d-1)}\nonumber\\
&=\|\rho\|_{L_y^2}^2\bg(\frac{1}{2}\|\nabla_x P\|_{L_x^2}^2-\frac{1}{2+4/(d-1)}\|P\|_{L_x^{2+4/(d-1)}}^{2+4/(d-1)}\bg)
=\|\rho\|_{L_y^2}^2\wm_{\|\rho\|_{L_y^2}^{-2}}<2\pi\wm_{(2\pi)^{-1}}.\label{star energy}
\end{align*}
Consequently,
\begin{align}
\lim_{\ld\to 0}m_{1,\ld}\leq \lim_{\ld\to 0}\mH_{\ld}(\psi)=\mH_*(\psi)<2\pi\wm_{(2\pi)^{-1}}
\end{align}
and \eqref{strictly small} follows.

Next, since the mapping $\ld\mapsto m_{1,\ld}$ is monotone increasing and $\ld_*$ is defined as an infimum, we know that $m_{1,\ld}<2\pi\wm_{(2\pi)^{-1}}$ for all $\ld\in(0,\ld_*)$. It is left to show the necessity of $y$-independence of the minimizers $u_\ld$ for $\ld>\ld_*$. We borrow an idea from \cite{GrossPitaevskiR1T1} to prove this claim. Assume the contrary that an optimizer $u_\ld$ of $m_{1,\ld}$ satisfies $\|\pt_y u_\ld\|_2^2\neq 0$. Let $\mu\in(\ld_*,\ld)$. Then
\begin{align*}
2\pi \wm_{(2\pi)^{-1}}=m_{1,\mu}\leq \mH_{\mu}(u_\ld)=\mH_{\ld}(u_\ld)+\frac{\mu-\ld}{2}\|\pt_y u_\ld\|_2^2<\mH_{\ld}(u_\ld)=m_{1,\ld}=2\pi \wm_{(2\pi)^{-1}},
\end{align*}
a contradiction. This completes the desired proof.
\end{proof}

\subsection{Proof of Theorem \ref{thm threshold mass}}
We now prove Theorem \ref{thm threshold mass} by using Lemma \ref{lemma auxiliary} and a simple rescaling argument.

\begin{proof}[Proof of Theorem \ref{thm threshold mass}]
From the proof of Lemma \ref{additional exitence lemma} we know that $m_c=c^{-1}m_{1,c^2}$. Using similar rescaling arguments we also know that $\wm_{(2\pi)^{-1}c}=c^{-1}\wm_{(2\pi)^{-1}}$. Set $c_*:=\sqrt{\ld_*}$, where $\ld_*$ is the number given by Lemma \ref{lemma auxiliary}. Then by Lemma \ref{lemma auxiliary} we know that
\begin{itemize}
\item For all $c\in(0,c_*)$
$$m_{c}=c^{-1}m_{1,c^2}<c^{-1}2\pi \wm_{(2\pi)^{-1}}=2\pi \wm_{(2\pi)^{-1}c}.$$

\item For all $c\in(c_*,\infty)$
$$m_{c}=c^{-1}m_{1,c^2}=c^{-1}2\pi \wm_{(2\pi)^{-1}}=2\pi \wm_{(2\pi)^{-1}c}.$$
\end{itemize}
The statement concerning the $y$-independence of the minimizers follows also from Lemma \ref{lemma auxiliary} simultaneously. That $m_{c_*}=2\pi \wm_{(2\pi)^{-1}c_*}$ follows from the continuity of the mappings $c\mapsto m_c$ and $c\mapsto \wm_c$ deduced from Lemma \ref{monotone lemma} and Lemma \ref{lem wmc property} respectively. This completes the desired proof.
\end{proof}

\section{Scattering and finite time blow-up below threshold}\label{sec: Scattering}
In this section we give the proofs of Theorem \ref{main thm} and \ref{thm blow up}. To begin with, we firstly give a quick recap of the large data scattering result for the focusing energy-critical NLS on $\R^4$. In the same subsection we also collect some useful energy-trapping results and a sharp Sobelev's inequality on $\R^3\times\T$ due to Yu-Yue-Zhao \cite{YuYueZhao2021}. Throughout this section we also denote by $\SSS=\SSS_{d+1}$ the best constant of the Sobolev's inequality on $\R^{d+1}$ defined by
$$ \SSS:=\inf_{u\in \mathcal{D}^{1,2}(\R^{d+1})}\bg(\|\nabla_{\R^{d+1}}u\|^2_{L^2(\R^{d+1})}\bg/\|u\|^2_{L^{2+4/(d-1)}(\R^{d+1})}\bg).$$

\subsection{Recap of the focusing energy-critical NLS on $\R^{d+1}$}
Consider the focusing energy-critical NLS
\begin{align}\label{rd nls}
i\pt_t u+\Delta_{\R^{4}}u=-|u|^{2}u\quad\text{on $\R^{4}$}.
\end{align}
We have the following large data scattering result for \eqref{rd nls} due to Dodson \cite{Dodson4dfocusing}:
\begin{theorem}[Threshold scattering of the focusing energy-critical NLS on $\R^{4}$, \cite{Dodson4dfocusing}]\label{theorem focusing dodson}
Assume that a function $\phi\in \dot{H}^1(\R^{4})$ satisfies
\begin{align}\label{threshold assumption rd}
\mH^*(\phi)<\frac{\SSS^{2}}{4}\quad\text{and}\quad\|\nabla_{\R^4}\phi\|_{L^2(\R^{4})}^2<\SSS^{2},
\end{align}
where the energy functional $\mH^*$ is defined by \eqref{def of h star}. Then a solution $u$ of \eqref{rd nls} with $u(0)=\phi$ is global and scattering in time.
\end{theorem}

We will also make use of the following energy-trapping result due to Kenig-Merle \cite{KenigMerle2006}.

\begin{lemma}[Energy-trapping, \cite{KenigMerle2006}]\label{kenig merle coercive}
Let $\phi\in \dot{H}^1(\R^{d+1})$. Assume that
$$\|\nabla_{\R^{d+1}}\phi\|_{L^2(\R^{d+1})}^2<\SSS^{\frac{d+1}{2}}$$
and there exists some $\delta_0\in(0,1)$ such that
$$\mH^*(\phi)\leq (1-\delta_0)\SSS^{\frac{d+1}{2}}/(d+1).$$
Then there exists some $\delta=\delta(\delta_0)\in(0,1)$ such that $\|\nabla_{\R^{d+1}}\phi\|_{L^2(\R^{d+1})}^2\leq (1-\delta)\SSS^{\frac{d+1}{2}}$.
\end{lemma}

In our context we indeed make use of a variational setting based on the semivirial functional $\mK(u)$ and the $\R^{d+1}$-virial functional $\mK^*(\phi)$ (defined by \eqref{def of k star}) that seem irrelevant to the scattering scheme \eqref{threshold assumption rd} at the first glance. The following lemma reveals the fact that the positivity of $\mK^*(\phi)$ is a sufficient condition for \eqref{threshold assumption rd}.

\begin{lemma}\label{energy trapping 1}
Assume that a function $\phi\in \dot{H}^1(\R^{d+1})$ satisfies $\mH^*(\phi)<\frac{\SSS^{\frac{d+1}{2}}}{d+1}$. If $\mK^*(\phi)>0$, then \eqref{threshold assumption rd} holds for $\phi$.
\end{lemma}

\begin{proof}
Assume therefore the contrary that $\|\nabla_{\R^{d+1}}\phi\|_{L^2(\R^{d+1})}^2\geq \SSS^{\frac{d+1}{2}}$. Using $\mK^*(\phi)>0$ we obtain
\begin{align*}
\frac{\SSS^{\frac{d+1}{2}}}{d+1}>\mH^*(\phi)=\frac{1}{2}\|\nabla_{\R^{d+1}}\phi\|_{L^2(\R^{d+1})}^2
-\frac{d-1}{2(d+1)}\|\phi\|_{L^{2+4/(d-1)}(\R^{d+1})}^{2+4/(d-1)}
>\frac{\|\nabla_{\R^{d+1}}\phi\|_{L^2(\R^{d+1})}^2}{d+1}\geq \frac{\SSS^{\frac{d+1}{2}}}{d+1},
\end{align*}
a contradiction.
\end{proof}

For the upcoming proofs we also need the following useful sharp Sobolev's inequality on $\R^3\times\T$ given by Yu-Yue-Zhao \cite{YuYueZhao2021}.
\begin{lemma}[Sharp Sobolev's inequality on $\R^3\times\T$, \cite{YuYueZhao2021}]\label{r1t3 sobolev lemma}
Let $d=3$. Then there exists some $c>0$ such that
\begin{align*}
\|u\|_{4}\leq \SSS^{-\frac12}\|\nabla_{x,y}u\|_2+c\|u\|_2.
\end{align*}
\end{lemma}

Combining with Young's inequality we deduce the following corollary of Lemma \ref{r1t3 sobolev lemma}.

\begin{corollary}\label{r1t3 sobolev lemma cor}
Let $d=3$. For any $\vare>0$ there exists some $C_\vare>0$ such that
\begin{align*}
\|u\|_{4}^4\leq (\SSS^{-2}+\vare)\|\nabla_{x,y}u\|_2^4+C_\vare\|u\|_2^4.
\end{align*}
\end{corollary}

\subsection{A refined uniform bound for $m_c$}
We note that in order to apply the scattering and energy-trapping arguments stated in last subsection, it is always assumed that the energy $\mH^*(\phi)$ must be strictly less than $\SSS^{\frac{d+1}{2}}/(d+1)$. Such a typical upper bound of the energy functional, which arises often in problems involving energy-critical potentials, can not be obtained by solely using Lemma \ref{cor lower bound 1} concerning the size of $m_c$. Indeed, we are able to prove that $m_c$ will never exceed the threshold $\SSS^{\frac{d+1}{2}}/(d+1)$.
\begin{lemma}[A refined upper bound for $m_c$]\label{lemma refined uniform bound}
We have $m_c\leq \mathcal{S}^{\frac{d+1}{2}}/(d+1)$ for any $c\in(0,\infty)$.
\end{lemma}

\begin{proof}
For $0<\vare\ll 1$, define
\begin{align}
v_{\varepsilon}(z) &:=\varphi(z)\cdot\bg(\frac{\varepsilon}{\varepsilon^2+|z|^2}\bg)^{\frac{d-1}{2}},
\label{eq 1}
\end{align}
where $\varphi\in C^\infty_c(\R^{d+1};[0,1])$ is a radially symmetric and decreasing cut-off function such that $\varphi(z)\equiv 1$ in $|z|\leq 1$ and $\varphi(z)\equiv 0$ in $|z|\geq 2$. Next, for $t>0$ we define the operator $S_t:H^1(\R^{d+1})\to H^1(\R^{d+1})$ by
\begin{align*}
S_t u(z):=t^{\frac{d+1}{2}}u(tz).
\end{align*}
Particularly, we have $\|S_t u\|_{L^2(\R^{d+1})}=\|u\|_{L^2(\R^{d+1})}$ for any $t\in(0,\infty)$. Moreover, by \cite[Lem. 5.3]{SoaveCritical} we can find a unique $t_\vare^*=t_\vare^*(v_\vare)\in(0,\infty)$ such that
$$\mK^*(S_{t_\vare^*}v_\vare)=0\quad\text{and}\quad\mH^*(S_{t_\vare^*}v_\vare)=\SSS^{\frac{d+1}{2}}/(d+1)+O(\vare^{d-1}).$$
Next, we define the operator $S^s:H^1(\R^{d+1})\to H^1(\R^{d+1})$ for $s\in(0,\infty)$ by
\begin{align*}
S^s u(z):=s^{\frac{d-1}{2}}u(sz).
\end{align*}
One easily verifies that
\begin{align*}
\mK^*(S^s S_{t_\vare^*}v_\vare)=\mK^*(S_{t_\vare^*}v_\vare)=0\quad\text{and}\quad
\mH^*(S^s S_{t_\vare^*}v_\vare)=\mH^*(S_{t_\vare^*}v_\vare)
\end{align*}
for all $s\in (0,\infty)$, $\,S^s S_{t_\vare^*}v_\vare(z)$ is supported in $|z|\leq 1$ for all sufficiently large $s$ and
\begin{align*}
\lim_{s\to\infty}\|S^s S_{t_\vare^*}v_\vare\|_{L^2(\R^{d+1})}=0.
\end{align*}
Hence for all sufficiently large $s$, writing $z$ as $z=(x,y)$ we may identify $S^s S_{t_\vare^*}v_\vare$ as a function in $H_{x,y}^1$ by extending $S^s S_{t_\vare^*}v_\vare$ periodically modulo $2\pi$ along the $y$-direction. In particular, we have
\begin{align*}
\mM(S^s S_{t_\vare^*}v_\vare)=\|S^s S_{t_\vare^*}v_\vare\|^2_{L^2(\R^{d+1})}\quad\text{and}\quad
\mH(S^s S_{t_\vare^*}v_\vare)=\mH^*(S^s S_{t_\vare^*}v_\vare).
\end{align*}
Since $S^s S_{t_\vare^*}v_\vare$ is radially symmetric on $\R^{d+1}$, we also infer that
\begin{align*}
\mK(S^s S_{t_\vare^*}v_\vare)&=\frac{d}{d+1}\mK^*(S^s S_{t_\vare^*}v_\vare)=0.
\end{align*}
For a given $c\in(0,\infty)$, we choose some $s$ sufficiently large such that $\mM(S^s S_{t_\vare^*}v_\vare)=:\tilde{c}\leq c$. Using the monotonicity of the mapping $c\mapsto m_c$ deduced from Lemma \ref{monotone lemma} we infer that
\begin{align*}
m_c\leq m_{\tilde{c}}\leq \mH(S^s S_{t_\vare^*}v_\vare)=\mH^*(S_{t_\vare^*}v_\vare)=\SSS^{\frac{d+1}{2}}/(d+1)+O(\vare^{d-1}).
\end{align*}
Since $\vare$ can be chosen arbitrarily small, we conclude the desired claim.
\end{proof}

\subsection{Function spaces and Strichartz estimates}\label{sec littlewood paley}
In this subsection we introduce the function spaces and Strichartz estimates for the model problem \eqref{nls} (in the case $d=3$) which were originated in \cite{HerrTataruTz1,HerrTataruTz2,HaniPausader,RmT1}.
We begin with defining the Littlewood-Paley projectors. Let $\Phi\in C_c^\infty(\R;[0,1])$ be radially symmetric and decreasing, $\Phi(t)\equiv 1$ for $|t|\leq 1$ and $\Phi(t)\equiv 0$ for $|t|\geq 2$. For $z\in\R^4$ let $\eta(z):=\Phi(z_1)\Phi(z_2)\Phi(z_3)\Phi(z_4)$. Then for $N>0$ we define
\begin{align*}
\eta_{\leq N}(z):=\eta(z/N),\quad\eta_{N}(z):=\eta_{\leq N}(z)-\eta_{\leq N/2}(z),\quad \eta_{> N}(z):=1-\eta_{\leq N/2}(z).
\end{align*}
For a dyadic number $N\leq 1$ we define the Littlewood-Paley projector $P_{\leq N}$ by $\mathcal{F}(P_{\leq N})=\eta_1$. For $N\geq 2$ we similarly define
\begin{align*}
\mathcal{F}(P_{\leq N})=\eta_{\leq N},\quad \mathcal{F}(P_{N})=\eta_{N},\quad\mathcal{F}(P_{>N})=\eta_{>N}.
\end{align*}

Next, we introduce the spaces $X^s$ and $Y^s$.
Denote by $C=(-1/2,1/2]^4\in \R^4$ the unit cube in $\R^4$. For $z\in\R^4$ the translated cube $C_z$ is defined by $C_z:=C+z$. Moreover, we define the projector $P_{C_z}$ by
$$\mathcal{F}(P_{C_z}u):=\chi_{C_z}\mathcal{F}(u),$$
where $\chi_{C_z}$ is the characteristic function of $C_z$. For $s\in\R$ we then define the spaces $X_0^s(\R)$ and $Y^s(\R)$ through the norms
\begin{align*}
\|u\|_{X^s_0(\R)}^2:=\sum_{z\in\Z^4}\la z\ra^{2s}\|P_{C_z} u\|_{U_{\Delta_{x,y}}^2(\R;L^2_{x,y})},\\
\|u\|_{Y^s(\R)}^2:=\sum_{z\in\Z^4}\la z\ra^{2s}\|P_{C_z} u\|_{V_{\Delta_{x,y}}^2(\R;L^2_{x,y})},
\end{align*}
where $U_{\Delta_{x,y}}^2$ and $V_{\Delta_{x,y}}^2$ are the standard atom spaces taking values in $L_{x,y}^2$ (see for instance \cite{HadacHerrKoch2009} for their precise definitions). For any subinterval $I\subset\R$, the space $X^s(I)$ is defined through the norm
\begin{align*}
\|u\|_{X^s(I)}:=\inf\{\|v\|_{X^s_0(\R)}:v\in X_0^s(\R),\,v|_I=u|_I\}.
\end{align*}
The space $Y^s(I)$ is similarly defined. We also define the space $X^s_c(\R)$ by
\begin{align*}
X^s_c(\R):=\{u\in C(\R;H_{x,y}^s):\phi_{-\infty}:=\lim_{t\to -\infty}e^{-it\Delta_{x,y}}u(t)\text{ exists in $H_{x,y}^s$ and
 $u(t)-e^{it\Delta_{x,y}}\phi_{-\infty}\in X^s_0(\R)$}\}.
\end{align*}
The space $X^s_c(\R)$ is equipped with the norm
\begin{align*}
\|u\|_{X^s_c(\R)}^2:=\|\phi_{-\infty}\|_{H_{x,y}^s}^2+\|u(t)-e^{it\Delta_{x,y}}\phi_{-\infty}\|^2_{X_0^s(\R)}\sim \sup_{K\subset \R\,\text{compact}}
\|u\|_{X(K)}^2.
\end{align*}
For any subinterval $I\subset\R$, we define the space $X^s_c(I)$ via the second definition of the $X_c^s$-norm running over all compact subsets $K$ of $I$. We also define the space $X^s_{c,\mathrm{loc}}(I)$ by
\begin{align}\label{def of Xcloc}
X^s_{c,\mathrm{loc}}(I):=\bigcap_{J\subset I\,\text{compact}}X^s_{c}(J).
\end{align}
For an interval $I=(a,b)$, the space $N^s(I)$ is defined through the norm
\begin{align*}
\|u\|_{N^s(I)}:=\|\int_a^t e^{i(t-\sigma)\Delta_{x,y}}u(\sigma)\,d\sigma\|_{X^s(I)}.
\end{align*}
When $s=1$, we simply write $X^1=X$, $Y^1=Y$ and so on. We record the following useful properties of the previously defined function spaces.

\begin{lemma}[Embeddings between function spaces, \cite{HadacHerrKoch2009,HerrTataruTz1}]
For any $s\in\R$ and $p\in(2,\infty)$ we have
\begin{align*}
U^2_{\Delta_{x,y}}(I;H^s_{x,y})\hookrightarrow X^s(I)\hookrightarrow Y^s(I)\hookrightarrow V_{\Delta_{x,y}}^2(I;H_{x,y}^s)
\hookrightarrow U^p_{\Delta_{x,y}}(I;H_{x,y}^s)\hookrightarrow L_t^\infty(I;H_{x,y}^s).
\end{align*}
\end{lemma}

\begin{lemma}[Duality of $N$ and $Y^{-1}$, \cite{HerrTataruTz1}]
For $u\in L_t^1H_{x,y}^1(I)$ we have
\begin{align*}
\|u\|_{N(I)}\lesssim \sup_{\|v\|_{Y^{-1}(I)}\leq 1}\int_{I\times(\R^3\times\T)}u(t,x,y)\bar{v}(t,x,y)\,dxdydt.
\end{align*}
Moreover, for any smooth function $g$ defined on $I=[a,b]$ we have
\begin{align*}
\|g\|_{X(I)}\lesssim \|g(a)\|_{H_{x,y}^1}+\bg(\sum_{N\geq 1}\|P_N(i\pt_t+\Delta_{x,y})g\|_{L_t^1H_{x,y}^1(I)}^2\bg)^{\frac12}.
\end{align*}
\end{lemma}

We shall also need the following Strichartz estimate on $\R^3\times\T$:

\begin{lemma}[Strichartz estimate on $\R^3\times\T$, \cite{RmT1,Barron}]\label{thm barron}
Let $p\in(3,6)$ and $\frac{2}{q}+\frac{3}{p}=\frac32$. Then for any $\phi\in H_{x,y}^1$ we have
\begin{align*}
\bg(\sum_{\gamma\in\Z}\|e^{it\Delta_{x,y}}\phi\|^q_{L_{t,x,y}^p(\gamma-1,\gamma+1)}\bg)^{\frac1q}\lesssim \|\phi\|_{H_{x,y}^{2-6/p}}.
\end{align*}
\end{lemma}

Using the embedding $\ell^{q}\hookrightarrow \ell^{\tilde{q}}$ for $\tilde{q}>q$ we obtain immediately the following corollary of Lemma \ref{thm barron}.

\begin{corollary}\label{cor barron}
Let $p\in(3,6)$ and $\frac{2}{q}+\frac{3}{p}=1$. Then for any $\phi\in H_{x,y}^1$ we have
\begin{align*}
\bg(\sum_{\gamma\in\Z}\|e^{it\Delta_{x,y}}\phi\|^q_{L_{t,x,y}^p(\gamma-1,\gamma+1)}\bg)^{\frac1q}\lesssim \|\phi\|_{H_{x,y}^{2-6/p}}.
\end{align*}
\end{corollary}

\subsection{Small data and stability theories}
We collect in this subsection the small data and stability theories for \eqref{nls} that were originally given in \cite{RmT1}, where the defocusing analogue of \eqref{nls} was studied. We shall begin with defining the scattering $Z$-norm. For a time interval $I\subset\R$, we define the space $Z(I)$ via the norm
\begin{align*}
\|u\|_{Z(I)}:=\bg(\sum_{N\geq 1}N^{6-p}\bg(\sum_{\gamma\in\Z}\|\chi_I(t)P_Nu\|^q_{L_{t,x,y}^p(\gamma-1,\gamma+1)}\bg)^{\frac{p}{q}}\bg)^{\frac1p},
\end{align*}
where $2/q+3/p=1$ with $p\in (5,11/2)$ . The precise value of $p$ is nonetheless of no importance. Notice also that by Corollary \ref{cor barron} we have for any $\phi\in H_{x,y}^1$
\begin{align}\label{upper bound zi norm}
\|e^{it\Delta_{x,y}}\phi\|_{Z(\R)}\lesssim \|\phi\|_{H_{x,y}^1}.
\end{align}

We are now ready to state the small data and stability results for \eqref{nls} due to Zhao \cite{RmT1}.
\begin{lemma}[Small data theory, \cite{RmT1}]\label{lemma small data}
Let $I$ be a time interval containing zero. Suppose that a function $u_0\in H_{x,y}^1$ satisfies
$$\|u_0\|_{H_{x,y}^1}\leq A.$$
Then there exists some $\delta=\delta(A)$ such that if
$$\|e^{it\Delta_{x,y}}u_0\|_{Z(I)}\leq \delta,$$
then \eqref{nls} possesses a unique strong solution $u$ in $X_c(I)$ with $u(0)=u_0$. Moreover, if a solution $u\in X_{c,\mathrm{loc}}(I)$ satisfies
$$\|u\|_{Z(I)}<\infty,$$
then
\begin{itemize}
\item[(i)] If $I$ is finite, then $u$ can be extended to some strictly larger interval $J$ with $I\subsetneq J \subset\R$.

\item[(ii)] If $I$ is infinite, then $u\in X_c(I)$.
\end{itemize}
\end{lemma}

\begin{remark}\label{small remark}
Using \eqref{upper bound zi norm} we infer that \eqref{nls} is always globally well-posed and scattering in time when $\|u_0\|_{H_{x,y}^1}$ is sufficiently small.
\end{remark}

\begin{lemma}[Stability theory, \cite{RmT1}]\label{lem stability cnls}
Let $I\subset\R$ be an interval containing zero and let $\tilde{u}\in X(I)$ be an approximate solution of the perturbed NLS
\begin{align*}
i\pt_t \tilde{u}+\Delta_{x,y}u=-|\tilde{u}|^2 \tilde{u}+e
\end{align*}
with some error term $e$. Suppose also that there exists some $M>0$ such that
\begin{align*}
\|\tilde{u}\|_{Z(I)}+\|\tilde{u}\|_{L_t^\infty H_{x,y}^1(I)}\leq M.
\end{align*}
Then there exists some positive $\vare_0=\vare_0(M)\ll 1$ such that if
\begin{align*}
\|\tilde{u}(0)-u_0\|_{H_{x,y}^1}+\|e\|_{N(I)}\leq \vare<\vare_0,
\end{align*}
then there exists a solution $u\in X(I)$ of \eqref{nls} with $u(0)=u_0$ and
\begin{align*}
\|u\|_{X(I)}\leq C(M)\quad\text{and}\quad
\|u-\tilde{u}\|_{X(I)}\leq C(M)\vare.
\end{align*}
\end{lemma}

\subsection{Linear profile decomposition}
The present subsection is devoted to introducing a suitable profile decomposition for the model problem \eqref{nls}, which being a standard preliminary for a rigidity proof based on the concentration compactness arguments. We shall invoke the profile decomposition applied in \cite{RmT1} for the study of the defocusing analogue of \eqref{nls}, which follows the same fashion as the ones given in \cite{hyperbolic,Ionescu1,Ionescu2,HaniPausader}.

We firstly fix some necessary notation. Let $\eta\in C_c^\infty(\R^4;[0,1])$ be the same auxiliary function given previously for constructing the Littlewood-Paley projectors (see the beginning of Section \ref{sec littlewood paley}). For a function $\phi\in \dot{H}^1(\R^4)$ and a number $N\geq 1$ we define the function $\phi_N$ by
\begin{align*}
\phi_N(z):=N\eta(N^{\frac12}z)\phi(Nz).
\end{align*}
Let now $\Psi:\{z\in\R^4:|z|\leq 1\}\to\R^3\times\T$ be the identity mapping. Then we define the function $f_N(z)$ for $z=(x,y)\in\R^3\times\T$ by
\begin{align*}
f_N(z):=\phi_N(\Psi^{-1}(z)).
\end{align*}
The definition of $f_N$ is at the first glance somewhat misunderstanding, since $\Psi^{-1}(z)$ is not well-defined for arbitrary points $z\in \R^3\times \T$. We shall make the following convention to clarify the definition of $f_N$:
\begin{itemize}
\item For $z\in\R^3\times\T$, we identify $z$ as the point locating at $\R^3\times [-\pi,\pi]$.
\item For $z\in \R^3\times [-\pi,\pi]$, if $|z|>1$, then we simply set $f_N(z)=0$.
\end{itemize}
In other words, for all sufficiently large $N$ (which will be the case for a Euclidean profile) $f_N$ is nothing else but the periodic extension of $\phi_N$ along the $y$-direction modulo $2\pi$. Using Sobolev's embedding and H\"older's inequality one easily verifies that $f_N\in H^1_{x,y}$ and
\begin{align*}
\limsup_{N\to\infty}\|f_N\|_{H_{x,y}^1}\lesssim \|\phi\|_{\dot{H}^1(\R^4)}.
\end{align*}
For our purpose we will need the following stronger statement on the asymptotics of $f_N$.
\begin{lemma}\label{fn to phi}
We have $\|f_N\|_2^2=o_N(1)$ and
\begin{align*}
\|\phi\|^2_{\dot{H}^1(\R^4)}&=\|f_N\|^2_{\dot{H}_{x,y}^1}+o_N(1),\\
\|\phi\|^4_{L^4(\R^4)}&=\|f_N\|_{4}^4+o_N(1)
\end{align*}
as $N\to\infty$.
\end{lemma}
\begin{proof}
We only prove the statements concerning the $L_{x,y}^2$- and $\dot{H}_{x,y}^1$-norms of $f_N$, the one for the $L_{x,y}^4$-norm can de deduced similarly. Notice that when $N$ tends to infinity the function $f_N$ concentrates to the zero point, thus any integration over $\R^3\times\T$ can be replaced to an integration over $\R^4$ for all sufficiently large $N$. For $\|f_N\|^2_2$, using change of variable, H\"older and Sobolev's embedding we obtain
\begin{align*}
\|f_N\|_2^2&= N^2\int_{\R^4}|\eta(N^{\frac12}z)\phi(Nz)|^2\,dz
=N^{-2}\int_{\R^4}|\eta(N^{-\frac12}z)\phi(z)|^2\,dz\nonumber\\
&\lesssim N^{-2}\bg(\int_{\R^4}|\eta(N^{-\frac12}z)|^4\,dz\bg)^{\frac12}\|\phi\|_{L^4(\R^4)}^2
\lesssim N^{-1}\|\phi\|_{\dot{H}^1(\R^4)}^2=o_N(1).
\end{align*}
Next, by product rule we know
\begin{align*}
\nabla_{z}\phi_N=N^{\frac32}\nabla_z\eta(N^{\frac12}z)\phi(Nz)+N^{2}\eta(N^{\frac12}z)\nabla_z\phi(Nz)=:I+II.
\end{align*}
By dominated convergence theorem we already infer that $\|II\|_{L^2(\R^4)}^2\to \|\phi\|_{\dot{H}^1(\R^4)}^2$. For $I$, we choose some $\tilde{\phi}\in C_c^\infty(\R^4)$ such that $\|\phi-\tilde{\phi}\|_{\dot{H}^1(\R^4)}\leq\vare$, where $\vare>0$ is some arbitrarily chosen constant. Then using H\"older
\begin{align*}
\|I-N^{\frac32}\nabla_z\eta(N^{\frac12}z)\tilde{\phi}(Nz)\|^2_{L^2(\R^4)}\lesssim \|\phi-\tilde{\phi}\|_{L^4(\R^d)}^2\leq \vare^2.
\end{align*}
Since $\vare$ is arbitrarily chosen, it suffices to show
\begin{align*}
N^3\int_{\R^4}|\nabla_z\eta(N^{\frac12}z)|^2|\tilde{\phi}(Nz)|^2\,dz=o_N(1).
\end{align*}
But using change of variable and the uniform boundedness of $\nabla_z\eta$ we obtain
\begin{align*}
N^3\int_{\R^4}|\nabla_z\eta(N^{\frac12}z)|^2|\tilde{\phi}(Nz)|^2\,dz
\lesssim N^{-1}\int_{\R^4}|\tilde{\phi}(z)|^2\,dz=o_N(1),
\end{align*}
as desired.
\end{proof}

Next, for $(f,t_0,p_0)\in L_{x,y}^2\times\R\times(\R^3\times\T)$ and $\phi \in \dot{H}^1(\R^4)$ we define the operators
\begin{align*}
\pi_{p_0}f&:=f(z-p_0),
\\
\Pi_{t_0,p_0}f&:= e^{-it_0\Delta}f(z-p_0)=\pi_{p_0}(e^{-it_0\Delta}f),
\\
\mathcal{T}_N \phi&:=f_N(\phi)=\phi_N(\Psi^{-1}(z)).
\end{align*}
We are now ready to introduce the concepts of \textit{Euclidean} and \textit{scale-one} profiles

\begin{definition}[Frames and profiles]
We define a \textbf{frame} $\mathcal{F}$ to be a sequence $(N_n,t_n,p_n)_n$ in $2^{\N_0}\times \R\times(\R^3\times\T)$. We also define two special classes of frames:
\begin{itemize}
\item A \textbf{Euclidean frame} $\mathcal{F}$ is a frame satisfying
\begin{itemize}
\item[(i)]$\lim_{n\to\infty}N_n=\infty$.
\item[(ii)]$t_n\equiv0$ for all $n\in\N$ or $\lim_{n\to\infty} |t_n N_n^2| =\infty$.
\end{itemize}
\item A \textbf{scale-one frame} $\mathcal{F}$ is a frame satisfying
\begin{itemize}
\item[(i)]$N_n\equiv 1$ for all $n\in\N$.
\item[(ii)]$t_n\equiv0$ for all $n\in\N$ or $\lim_{n\to\infty} |t_n| =\infty$.
\end{itemize}
\end{itemize}
Associated to each Euclidean or scale-one frame, we define a \textbf{profile} as follows:
\begin{itemize}
\item If $\mathcal{F}$ is a Euclidean frame, then for $\phi\in \dot{H}^1(\R^4)$ we define the \textbf{Euclidean profile} $T_n \phi$ by
\begin{align*}
T_n\phi:=\Pi_{t_n,p_n}\mathcal{T}_{N_n}\phi.
\end{align*}
\item If $\mathcal{F}$ is a scale-one frame, then for $\phi\in H_{x,y}^1$ we define the \textbf{scale-one profile} $T_n \phi$ by
\begin{align*}
T_n\phi:=\Pi_{t_n,p_n}\phi.
\end{align*}
\end{itemize}
\end{definition}

\begin{remark}
In the rest of the paper, a frame will always be referred to as a Euclidean or a scale-one frame.
\end{remark}

We have the following linear profile decomposition for a bounded sequence in $H_{x,y}^1$ according to \cite{RmT1}. The version stated here is slightly different from the original one given in \cite{RmT1} and is better suited to our context.

\begin{lemma}[Linear profile decomposition, \cite{RmT1}]\label{linear profile}
Let $(\psi_n)_n$ be a bounded sequence in $H_{x,y}^1$. Then up to a subsequence, there exist nonzero $(\phi^j)_j\subset \dot{H}^1(\R^4)\cup H_{x,y}^1$, a sequence of frames $(N_n^j,t_n^j,p_n^j)_{j,n}$, a sequence of remainders $(w_n^j)_{j,n}\subset H_{x,y}^1$ and some number $K^*\in\N\cup\{\infty\}$ such that
\begin{itemize}
\item[(i)]For any finite $1\leq k\leq K^*$ we have the decomposition
\begin{align*}
\psi_n=\sum_{j=1}^k T_n^j\phi^j+w_n^k.
\end{align*}

\item[(ii)] The remainders $(w_n^j)_{j,n}$ satisfy
\begin{align*}
\lim_{j\to K^*}\lim_{n\to\infty}\|e^{it\Delta_{x,y}}w_n^j\|_{Z(\R)}=0.
\end{align*}

\item[(iii)] The frames are orthogonal in the sense that
\begin{align*}
|\log (N_n^j/N_n^k)|+|t_n^k-t_n^j|(N_n^j)^2+|p_n^k-p_n^j|N_n^j\to\infty
\end{align*}
as $n\to\infty$ for any $j\neq k$.

\item[(iv)] For any finite $1\leq k\leq K^*$ and $D\in\{1,\pt_{x_i},\pt_y\}$ we have the energy decompositions
\begin{align*}
\|D\psi_n\|_{2}^2&=\sum_{j=1}^k\|D(T_n^j\tdu^j)\|_{2}^2+\|Dw_n^k\|_{2}^2+o_n(1),
\\
\|\psi_n\|_{4}^{4}&=\sum_{j=1}^k\|T_n^j\tdu^j\|_{4}^{4}
+\|w_n^k\|_{4}^{4}+o_n(1)
.
\end{align*}
\end{itemize}
\end{lemma}

We end this subsection with the following small scale approximation result for a Euclidean profile.

\begin{lemma}[Small scale approximation]\label{small scale proxy lem lem}
Let $\mathcal{F}=(N_n,t_n,p_n)_n$ be a Euclidean frame and let $\phi\in \dot{H}^1(\R^4)$ satisfy \eqref{threshold assumption rd}. Let also $U_n$ be a solution of \eqref{nls} with $U_n(0)=\Pi_{t_n,p_n}\mathcal{T}_{N_n}\phi$. Then
\begin{itemize}
\item[(i)] For all sufficiently large $n$ the solutions $U_n$ are global and scattering. Moreover, we have
\begin{align*}
\limsup_{n\to\infty}\|U_n\|_{X(\R)}\lesssim_{\mH^*(\phi)}1.
\end{align*}

\item[(ii)] Let $u$ be the global solution of \eqref{rd nls} with $u(0)=\phi$ and let $\phi^{\pm}\in \dot{H}^1(\R^4)$ be the scattering data such that
\begin{align*}
\lim_{t\to\pm\infty}\|u(t)-e^{it\Delta_{\R^d}}\phi^{\pm}\|_{\dot{H}^1(\R^4)}=0.
\end{align*}
For $x\in \R^3\times\T$ and $R>0$ define
\begin{align*}
u_{n,R}(t,z):=\pi_{p_n}[N_n\,\eta (N_n\Psi^{-1}(z)/R)\,u(N_n^2(t-t_n),N_n\Psi^{-1}(z))].
\end{align*}
Then
\begin{align*}
\lim_{T\to\infty}\lim_{R\to\infty}\lim_{n\to\infty}\|U_n-u_{n,R}\|_{X(|t-t_n|\leq N_n^{-2}T)}&=0,\\
\lim_{T\to\infty}\lim_{n\to\infty}\|U_n-\Pi_{t-t_n,p_n}\mathcal{T}_{N_n}\phi^{\pm}\|_{X(|t-t_n|\geq N_n^{-2}T)}&=0.
\end{align*}
\end{itemize}
\end{lemma}

\begin{proof}
The proof is almost identical to the proof of \cite[Thm. 5.4]{RmT1} and \cite[Prop. 5.4]{HaniPausader}. The only different step is to replace the scattering result from \cite{defocusing4d} for the defocusing analogue of \eqref{rd nls} applied in the proof of \cite[Thm. 5.4]{RmT1} to Theorem \ref{theorem focusing dodson}. In order to apply Theorem \ref{theorem focusing dodson}, we shall therefore additionally assume that $\phi$ satisfies \eqref{threshold assumption rd}.
\end{proof}

\subsection{The MEI-functional and its properties}
In this subsection we introduce the mass-energy-indicator (MEI) functional $\mD$ and state some of its very useful properties which play a fundamental role for setting up an inductive hypothesis of a contradiction proof. The MEI-functional was firstly introduced in \cite{killip_visan_soliton} for the study of the focusing-defocusing 3D cubic-quintic NLS and further applied in \cite{ArdilaDipolar,killip2020cubicquintic,Luo_JFA_2022,Luo_DoubleCritical,Luo_Waveguide_MassCritical,CubicQuinticPotential,Luo_inter} for different models. Such inductive scheme is particularly useful when the inductive scheme is multidirectional (for instance in our case we need to consider the mass and energy separately). Since the proofs of the to be listed statements are identical to the ones given in \cite{Luo_inter}, we shall simply omit the details here.

To begin with, we firstly define the domain $\Omega\subset \R^2$ by
\begin{align*}
\Omega&:=\bg((-\infty,0]\times \R\bg)\cup\bg\{(c,h)\in\R^2:c\in(0,\infty),h\in(-\infty,m_c)\bg\},
\end{align*}
where $m_c$ is defined by \eqref{def of mc}. Then we define the MEI-functional $\mD:\R^2\to [0,\infty]$ by
\begin{align*}
\mD(c,h)=\left\{
             \begin{array}{ll}
             h+\frac{h+c}{\mathrm{dist}((c,h),\Omega^c)},&\text{if $(c,h)\in \Omega$},\\
             \infty,&\text{otherwise}.
             \end{array}
\right.
\end{align*}
For $u\in H_{x,y}^1$, define $\mD(u):=\mD(\mM(u),\mH(u))$. We also define the set $\mA$ by
\begin{align*}
\mA&:=\{u\in H_{x,y}^1:\mH(u)<m_{\mM(u)},\,\mK(u)>0\}.
\end{align*}
By conservation of mass and energy we know that if $u$ is a solution of \eqref{nls}, then $\mD(u(t))$ is a conserved quantity, thus in the following we simply write $\mD(u)=\mD(u(t))$ as long as $u$ is a solution of \eqref{nls}.

\begin{lemma}[Invariance of NLS-flow along $\mA$]\label{invariance from mA}
Let $u$ be a solution of \eqref{nls} and assume that there exists some $t$ in the lifespan of $u$ such that $u(t)\in\mA$. Then $u(t)\in\mA$ for all $t$ in the maximal lifespan of $u$.
\end{lemma}

\begin{remark}
In view of Lemma \ref{invariance from mA} we will therefore write $u\in\mA$ for a solution $u$ of \eqref{nls} if $u(t)\in\mA$ for some $t$ in the lifespan of $u$.
\end{remark}

\begin{lemma}[Equivalence of $\mH(u)$ and $\|\nabla_{x,y}u\|_2^2$]\label{lemma coercivity}
Let $u\in\mA$. Then
\begin{align*}
\frac1d\|\nabla_{x,y} u\|_2^2&\leq \mH(u)\leq\frac{1}{2}\|\nabla_{x,y} u\|_2^2
\end{align*}
\end{lemma}

\begin{lemma}[Properties of the MEI-functional]\label{cnls killip visan curve}
Let $u,u_1,u_2$ be functions in $H_{x,y}^1$. The following statements hold true:
\begin{itemize}
\item[(i)] $u\in\mA \Leftrightarrow\mD(u)\in(0,\infty)$.

\item[(ii)] Let $u_1,u_2\in \mA$ satisfy $\mM(u_1)\leq \mM(u_2)$ and $\mH(u_1)\leq \mH(u_2)$, then $\mD(u_1)\leq \mD(u_2)$. If in addition either $\mM(u_1)<\mM(u_2)$ or $\mH(u_1)<\mH(u_2)$, then $\mD(u_1)<\mD(u_2)$.

\item[(iii)] Let $\mD_0\in(0,\infty)$. Then
\begin{align}
m_{\mM(u)}-\mH(u)&\gtrsim_{\mD_0} 1\label{small of unaaa},\\
\mH(u)+\mM(u)&\lesssim_{\mD_0}\mD(u)\label{mei var2}
\end{align}
uniformly for all $u\in \mA$ with $\mD(u)\leq \mD_0$.
\end{itemize}
\end{lemma}

\subsection{Existence of a minimal blow-up solution}
Having all the preliminaries we are now ready to construct a minimal blow-up solution of \eqref{nls}. Define
\begin{align*}
\tau(\mD_0):=\sup\bg\{\|u\|_{Z(I_{\max})}:
\text{ $u$ is solution of \eqref{nls}, }\mD(u)\in (0,\mD_0)\bg\}
\end{align*}
and
\begin{align}\label{introductive hypothesis}
\mD^*&:=\sup\{\mD_0>0:\tau(\mD_0)<\infty\}.
\end{align}
By Lemma \ref{lemma small data}, \ref{lemma coercivity}, \ref{cnls killip visan curve} and Remark \ref{small remark} we know that $\mD^*>0$. Therefore we may simply assume $\mD^*<\infty$, relying on which we derive a contradiction. This in turn ultimately implies $\mD^*=\infty$ and the proof of Theorem \ref{main thm} will be complete in view of Lemma \ref{cnls killip visan curve}. By the inductive hypothesis we can find a sequence $(u_n)_n$ which are solutions of \eqref{nls} with $(u_n(0))_n\subset {\mA}$ and maximal lifespan $(I_{n})_n$ such that
\begin{gather}
\lim_{n\to\infty}\|u_n\|_{Z((\inf I_n,0])}=\lim_{n\to\infty}\|u_n\|_{Z([0, \sup I_n))}=\infty,\label{oo1}\\
\lim_{n\to\infty}\mD(u_n)=\mD^*.\label{oo2}
\end{gather}
Up to a subsequence we may also assume that
\begin{align*}
(\mM(u_n),\mH(u_n))\to(\mM_0,\mH_0)\quad\text{as $n\to\infty$}.
\end{align*}
By continuity of $\mD$ and finiteness of $\mD^*$ we know that
\begin{align*}
\mD^*=\mD(\mM_0,\mH_0),\quad
\mM_0\in[0,\infty),\quad
\mH_0\in[0,m_{\mM_0}).
\end{align*}
From Lemma \ref{lemma coercivity} and \ref{cnls killip visan curve} it follows that $(u_n(0))_n$ is a bounded sequence in $H_{x,y}^1$, hence Lemma \ref{linear profile} is applicable to $(u_n(0))_n$: There exist nonzero $(\phi^j)_j\subset \dot{H}^1(\R^4)\cup H_{x,y}^1$, a sequence of frames $(N_n^j,t_n^j,p_n^j)_{j,n}$, a sequence of remainders $(w_n^j)_{j,n}\subset H_{x,y}^1$ and some number $K^*\in\N\cup\{\infty\}$ such that
\begin{itemize}
\item[(i)]For any finite $1\leq k\leq K^*$ we have the decomposition
\begin{align}\label{cnls decomp}
u_n(0)=\sum_{j=1}^k T_n^j\phi^j+w_n^k.
\end{align}

\item[(ii)] The remainders $(w_n^k)_{k,n}$ satisfy
\begin{align}\label{cnls to zero wnk}
\lim_{k\to K^*}\lim_{n\to\infty}\|e^{it\Delta_{x,y}}w_n^k\|_{Z(\R)}=0.
\end{align}

\item[(iii)] The parameters are orthogonal in the sense that
\begin{align}\label{cnls orthog of pairs}
|\log (N_n^j/N_n^k)|+|t_n^k-t_n^j|(N_n^j)^2+|p_n^k-p_n^j|N_n^j\to\infty
\end{align}
as $n\to\infty$ for any $j\neq k$.

\item[(iv)] For any finite $1\leq k\leq K^*$ and $D\in\{1,\pt_{x_i},\pt_y\}$ we have the energy decompositions
\begin{align}
\|D(u_n(0))\|_{2}^2&=\sum_{j=1}^k\|D(T_n^j\tdu^j)\|_{2}^2+\|Dw_n^k\|_{2}^2+o_n(1),\label{orthog L2}\\
\|u_n(0)\|_{4}^{4}&=\sum_{j=1}^k\|T_n^j\tdu^j\|_{4}^{4}
+\|w_n^k\|_{4}^{4}+o_n(1)\label{cnls conv of h}.
\end{align}
\end{itemize}
We now define the nonlinear profiles as follows: Let $1\leq k\leq K^*$. If $(N_n^k,t_n^k,p_n^k)_n$ is a Euclidean frame, we define the nonlinear profile $u_n^k$ to be the solution $U_n$ given in Lemma \ref{small scale proxy lem lem} with $U_n(0)=T_n^k\phi^k$. If $(N_n^k,t_n^k,p_n^k)_n$ is a scale-one frame, then
\begin{itemize}
\item For $t^k_\infty=0$, we define $u^k$ as the solution of \eqref{nls} with $u^k(0)=\tdu^k$.

\item For $t^k_\infty\to\pm\infty$, we define $u^k$ as the solution of \eqref{nls} that scatters forward (backward) to $e^{it\Delta_{x,y}}\tdu^k$ in $H_{x,y}^1$.
\end{itemize}
In both cases we define the nonlinear profiles $u_n^k$ by
\begin{align*}
u_n^k:=u^k(t-t^k_n,z-p_n^k).
\end{align*}
Then $u_n^k$ is also a solution of \eqref{nls}. In all the cases we have for each finite $1\leq k \leq K^*$
\begin{align}\label{conv of nonlinear profiles in h1}
\lim_{n\to\infty}\|u_n^k(0)-T_n^k \tdu^k\|_{H_{x,y}^1}=0.
\end{align}

In the following we establish a Palais-Smale type lemma which is essential for the construction of the minimal blow-up solution.

\begin{lemma}[Palais-Smale-condition]\label{Palais Smale}
Let $(u_n)_n$ be a sequence of solutions of \eqref{nls} with maximal lifespan $I_n$, $u_n\in\mA$ and $\lim_{n\to\infty}\mD(u_n)=\mD^*$. Assume also that there exists a sequence $(t_n)_n\subset\prod_n I_n$ such that
\begin{align}\label{precondition}
\lim_{n\to\infty}\|u_n\|_{Z((\inf I_n,\,t_n])}=\lim_{n\to\infty}\|u_n\|_{Z([t_n,\,\sup I_n)}=\infty.
\end{align}
Then up to a subsequence, there exists a sequence $(x_n)_n\subset\R^3$ such that $(u_n(t_n, \cdot+x_n,y))_n$ strongly converges in $H_{x,y}^1$.
\end{lemma}

\begin{proof}
By time translation invariance we may assume that $t_n\equiv 0$. Let $(u_n^j)_{j,n}$ be the nonlinear profiles corresponding to the linear profile decomposition of $(u_n(0))_n$.
We divide the remaining proof into three steps.
\subsubsection*{Step 1: Decomposition of energies of the linear profiles}
We firstly show that for a given nonzero linear profile $\phi^j$ we have
\begin{align}
\mH(T_n^j\phi^j)&> 0,\label{bd for S}\\
\mK(T_n^j\phi^j)&> 0\label{pos of K}
\end{align}
for all sufficiently large $n=n(j)\in\N$. Since $\phi^j\neq 0$ we know that $T_n^j\phi^j\neq 0$ for all sufficiently large $n$. Suppose now that \eqref{pos of K} does not hold. Up to a subsequence we may assume that $\mK(T_n^j \phi^j)\leq 0$ for all sufficiently large $n$. Recall the energy functional $\mI$ defined by \eqref{def of mI}. Using \eqref{orthog L2} and \eqref{cnls conv of h} we infer that
\begin{align}
\mI(u_n(0))&=\sum_{j=1}^k\mI(T_n^j\tdu^j)+\mI(w_n^k)+o_n(1)\label{conv of i}.
\end{align}
By the non-negativity of $\mI$, \eqref{conv of i} and \eqref{small of unaaa} we know that there exists some sufficiently small $\delta>0$ depending on $\mD^*$ and some sufficiently large $N_1$ such that for all $n>N_1$ we have
\begin{align}\label{contradiction1}
\tm_{\mM(T_n^j\phi^j)}\leq\mI(T_n^j\phi^j)\leq \mI(u_n(0))+\delta
\leq\mH(u_n(0))+\delta\leq m_{\mM(u_n(0))}-2\delta,
\end{align}
where $\tm$ is the quantity defined by \eqref{mtilde equal m}. By continuity of $c\mapsto m_c$ we also know that for sufficiently large $n$ we have
\begin{align}\label{contradiction3}
m_{\mM(u_n(0))}-2\delta\leq m_{\mM_0}-\delta.
\end{align}
Using \eqref{orthog L2} we deduce that for any $\vare>0$ there exists some large $N_2$ such that for all $n>N_2$ we have
\begin{align*}
\mM(T_n^j\phi^j)\leq \mM_0+\vare.
\end{align*}
From the continuity and monotonicity of $c\mapsto m_c$ and Step 2 in the proof of Proposition \ref{thm existence of ground state 1}, we may choose some sufficiently small $\vare$ to see that
\begin{align}\label{contradiction2}
\tm_{\mM(T_n^j\phi^j)}=m_{\mM(T_n^j\phi^j)}\geq m_{\mM_0+\vare}\geq m_{\mM_0}-\frac{\delta}{2}.
\end{align}
Now \eqref{contradiction1}, \eqref{contradiction3} and \eqref{contradiction2} yield a contradiction. Thus \eqref{pos of K} holds, which combining with Lemma \ref{lemma coercivity} also yields \eqref{bd for S}. Similarly, for each $j\in\N$ we have
\begin{align}
\mH(w_n^j)&> 0,\label{bd for S wnj} \\
\mK(w_n^j)&> 0\label{pos of K wnj}
\end{align}
for sufficiently large $n$.
\subsubsection*{Step 2: Applicability of Lemma \ref{small scale proxy lem lem}}
Next, we show that for a Euclidean profile (say it is the $j$-th profile), the function $\phi^j\in\dot{H}^1(\R^4)$ given in the Euclidean profile satisfies
\begin{align*}
\mH^*(\phi^j)\leq \mH_0\in(0,\SSS^2/4)\quad\text{and}\quad\|\nabla_{\R^{4}}\phi^j\|_{L^2(\R^{4})}^2<\SSS^{2},
\end{align*}
which in turn implies that Lemma \ref{small scale proxy lem lem} is applicable for the $j$-th linear profile. That $\mH_0\in(0,\SSS^2/4)$ follows from Lemma \ref{lemma refined uniform bound}, Lemma \ref{cnls killip visan curve} and the fact that $\mD^*<\infty$. For $u\in H_{x,y}^1$, we define
\begin{align*}
\mK^{**}(u):=\|\nabla_{x,y}u\|_2^2-\|u\|_{4}^4.
\end{align*}
We aim to show that $\mK^{**}(T_n^j\phi^j)>0$ for all sufficiently large $n$. Assume that this is not the case. Up to a subsequence we may assume that $\mK^{**}(T_n^j\phi^j)\leq 0$ for all sufficiently large $n$. From Step 1 we already know that $\mK(T_n^j\phi^j)>0$ for $n\gg 1$. Therefore for all  $n\gg 1$ it is necessary that
$$\|\pt_y (T_n^j\phi^j)\|_2^2-\frac14\|T_n^j\phi^j\|_4^4<0.$$
We now recall the scaling operator $T_\ld$ defined by \eqref{def of t ld}. The term $T_\ld T_n^j\phi^j$ is well-defined for all $\ld\in(0,1]$ and all sufficiently large $n$ since $T_n^j\phi^j$ is concentrating to zero with shrinking support as $n\to\infty$. Let $\ld_*\in(0,1)$ satisfy
\begin{align*}
\|\pt_y (T_{\ld_*}T_n^j\phi^j)\|_2^2=\frac14\|T_{\ld_*}T_n^j\phi^j\|_4^4.
\end{align*}
By direct calculation one easily sees that $\ld_*=2\|T_n^j\phi^j\|_4^{-2}\|\pt_y(T_n^j\phi^j)\|_2<1$. Moreover, calculating the derivative of the mapping
$$\ld\mapsto g(\ld):=\ld^{-1}\|\pt_y T_n^j\phi^j\|_2^2+\frac{\ld}{4}\|T_n^j\phi^j\|_4^4$$
we see that $g(\ld)$ is monotone decreasing on $(0,\ld_*)$ and increasing on $(\ld_*,\infty)$. Next, we rewrite $\mH(T_{\ld}T_n^j\phi^j)$ to
\begin{align*}
\mH(T_{\ld}T_n^j\phi^j)=\frac{1}{2}g(\ld)+\frac{\ld}{2}\mK(T_n^j\phi^j).
\end{align*}
Since $\mK(T_n^j\phi^j)>0$ (for $n\gg 1$), we conclude that $\ld\mapsto \mH(T_{\ld}T_n^j\phi^j)$ is monotone increasing on $(\ld_*,1)$. On the other hand, by definition of $\ld_*$ we also know that
$$ \mK^{**}(T_{\ld_*}T_n^j\phi^j)=\ld_*\mK(T_n^j\phi^j)>0.$$
Thus there exists some $\ld_{**}\in(\ld_*,1]$ such that
\begin{align*}
\mH(T_{\ld_{**}}T_n^j\phi^j)\leq \mH(T_n^j\phi^j)\leq \mH_0<\SSS^2/4\quad\text{and}\quad
\mK^{**}(T_{\ld_{**}}T_n^j\phi^j)=0.
\end{align*}
Now let $\vare>0$ be given. By Corollary \ref{r1t3 sobolev lemma cor} and $\mK^{**}(T_{\ld_{**}}T_n^j\phi^j)=0$ we obtain
\begin{align*}
\|\nabla_{x,y}(T_{\ld_{**}}T_n^j\phi^j)\|_2^2=\|T_{\ld_{**}}T_n^j\phi^j\|_4^4\leq (\SSS^{-2}+\vare)\|\nabla_{x,y}(T_{\ld_{**}}T_n^j\phi^j)\|_2^4
+C_{\vare}\|T_{\ld_{**}}T_n^j\phi^j\|_2^4,
\end{align*}
which in turn implies
\begin{align*}
&\,(\SSS^{-2}+\vare)\|\nabla_{x,y}(T_{\ld_{**}}T_n^j\phi^j)\|_2^2\nonumber\\
\geq &\,1-\frac{C_\vare \|T_{\ld_{**}}T_n^j\phi^j\|_2^4}{\|\nabla_{x,y}(T_{\ld_{**}}T_n^j\phi^j)\|_2^2}=1-\frac{C_\vare \|T_n^j\phi^j\|_2^4}{\ld^3_{**}\|\nabla_{x}T_n^j\phi^j\|_2^2
+\ld_{**}\|\pt_yT_n^j\phi^j\|_2^2}
\nonumber\\
\geq&\ 1-\frac{C_\vare \|T_n^j\phi^j\|_2^4}
{\ld^3_{*}\|\nabla_{x}T_n^j\phi^j\|_2^2}
= 1-\frac{C_\vare \|T_n^j\phi^j\|_2^4\|T_n^j\phi^j\|_4^6}
{8\|\nabla_{x}T_n^j\phi^j\|_2^2\|\pt_yT_n^j\phi^j\|_2^3}.
\end{align*}
But by Lemma \ref{fn to phi} and the embedding $H_{x,y}^1\hookrightarrow L_{x,y}^4$ we know that
\begin{gather*}
\sup_{n\in\N}\|T_n^j\phi^j\|_4\lesssim 1,\nonumber\\
\lim_{n\to\infty}\|T_n^j\phi^j\|_2=0,\nonumber\\
\liminf_{n\to\infty}(\|\nabla_{x}T_n^j\phi^j\|_2^2\|\pt_yT_n^j\phi^j\|_2^3)=\|\nabla_{x}\phi^j\|_{L^2(\R^4)}^2\|\pt_{y} \phi^j\|_{L^2(\R^4)}^3>0.
\end{gather*}
Hence there exits some $J=J(\vare)$ such that for all $n\geq J$
\begin{align*}
(\SSS^{-2}+\vare)\|\nabla_{x,y}(T_{\ld_{**}}T_n^j\phi^j)\|_2^2\geq 1-\vare,
\end{align*}
or equivalently
\begin{align*}
\|\nabla_{x,y}(T_{\ld_{**}}T_n^j\phi^j)\|_2^2\geq \frac{1-\vare}{1+\vare \SSS^2}\SSS^2.
\end{align*}
Now combining with $\mK^{**}(T_{\ld_{**}}T_n^j\phi^j)=0$ we deduce
\begin{align*}
\frac{\SSS^2}{4}>\mH_0\geq \mH(T_{\ld_{**}}T_n^j\phi^j)=\frac{\|\nabla_{x,y}(T_{\ld_{**}}T_n^j\phi^j)\|_2^2}{4}\geq \frac{1-\vare}{4(1+\vare \SSS^2)}\SSS^2.
\end{align*}
But
$$ \lim_{\vare\to 0}\frac{1-\vare}{4(1+\vare \SSS^2)}\SSS^2=\frac{\SSS^2}{4}.$$
Hence we can choose $\vare\ll 1$ such that $\frac{1-\vare}{4(1+\vare \SSS^2)}\SSS^2$ lies between $\mH_0$ and $\SSS^2/4$, which leads to a contradiction by taking $n$ sufficiently. We thus conclude that $\mK^{**}(T_n^j \phi^j)>0$ for all sufficiently large $n$.

Notice that this is still not the desired claim. In the case $t_n\equiv 0$ the claim follows already from Lemma \ref{energy trapping 1} and Lemma \ref{fn to phi}. It is therefore left to consider the case $|t_n N_n^2 |\to\infty$. The main issue here is that the Schr\"odinger group does not necessarily leave the Lebesgue norm invariant. To overcome this difficulty we shall appeal to the Sobolev's inequality on $\T$ and the dispersive estimate on $\R^3$. For $\vare>0$ let $\tilde{\phi}\in C_c^\infty(\R^4)$ such that $\|\phi^j-\varphi\|_{\dot{H}^1(\R^4)}\leq\vare$. W.l.o.g we may assume that $\varphi$ takes the form $\varphi(x,y)=\varphi^1(x)\varphi^2(y)$ with $\varphi^1\in C_c^\infty(\R^3)$ and $\varphi^2\in C_c^\infty(\R)$. Indeed, the general form of $\varphi$ should be a finite sum of atoms taking the form $\varphi^1\varphi^2$, but once the claim for a single atom is proved, the general claim follows immediately by using the triangular inequality. Since spatial translations leave the Lebesgue norm invariant, we may also assume that $p_n^j\equiv 0$. Next we recall that the function $\eta$ is given as a product of the functions $\Phi$:
$$\eta(z)=\Phi(x_1)\Phi(x_2)\Phi(x_3)\Phi(y).$$
With slight abuse of notation we simply write $\Phi(x):=\Pi_{j=1}^3\Phi(x_j)$. Then
\begin{align*}
e^{it^j_n\Delta_{x,y}}(\eta((N_n^j)^{\frac12}z)\varphi(N_n^jz))=e^{it_n^j\Delta_x}(\Phi((N_n^j)^{\frac12}x)\varphi^1(N_n^jx))\times
e^{it_n^j\Delta_y}(\Phi((N_n^j)^{\frac12}y)\varphi^2(N_n^jy)).
\end{align*}
Consequently,
\begin{align*}
\|T_n^j\varphi\|_4^4&=(N_n^j)^4\|e^{it_n^j\Delta_x}(\Phi((N_n^j)^{\frac12}x)\varphi^1(N_n^jx))\|_{L^4(\R^3)}^4
\|e^{it_n^j\Delta_y}(\Phi((N_n^j)^{\frac12}y)\varphi^2(N_n^jy))\|_{L^4(\T)}^4\nonumber\\
&=:(N_n^j)^4\times I\times II.
\end{align*}
For $I$, using the dispersive estimate of $e^{it\Delta_x}$ on $\R^3$ we infer that
\begin{align*}
I\lesssim |t_n^j|^{-3}\|\Phi((N_n^j)^{\frac12}x)\varphi^1(N_n^jx)\|_{L^{\frac{4}{3}}(\R^3)}^4\lesssim |t_n^j|^{-3}(N_n^j)^{-9}\|\varphi^1\|_{L^{\frac{4}{3}}(\R^3)}^4.
\end{align*}
For $II$, we set $u:=\Phi((N_n^j)^{\frac12}y)\varphi^1(N_n^jy)$ and then decompose $u$ into
\begin{align*}
e^{it_n^j\Delta_y}u=m(u)+e^{it_n^j\Delta_y}(u-m(u)).
\end{align*}
For $m(u)$, we have
\begin{align*}
\|m(u)\|_{L^4(\T)}^4\lesssim |m(u)|^4\lesssim\bg(\int_{\R}|\varphi^2(N_n^j y)|\,dy\bg)^4=(N_n^j)^{-4}\|\varphi^2\|^4_{L^1(\R)}\lesssim
(N_n^j)^{-1}\|\varphi^2\|^4_{L^1(\R)}.
\end{align*}
For $e^{it_n^j\Delta_y}(u-m(u))$, using \eqref{sobolev torus 1}, the fact that $e^{it_n^j\Delta_y}$ is an isometry on $H_y^s$ with $s\in\R$ and interpolation we obtain
\begin{align*}
\|e^{it_n^j\Delta_y}(u-m(u))\|_{L^4(\T)}^4\lesssim \|u\|_{L^2(\T)}^3\|\pt_y u\|_{L^2(\T)}
\lesssim (N_n^j)^{-1}\|\varphi^2\|_{H^1(\R)}^4.
\end{align*}
Summing up, we infer by combining $t_n^j(N_n^j)^2\to\pm\infty$ that
\begin{align*}
\|T_n^j\varphi\|_4^4\lesssim |t_n^j(N_n^j)^2|^{-3}\|\varphi^1\|^4_{L^{\frac{4}{3}}(\R^3)}
\|\varphi^2\|^4_{H^1(\R)}\to 0
\end{align*}
as $n\to\infty$. Since $\vare$ can be chosen arbitrarily, we finally conclude $\|T_n^j\phi^j\|_4^4=o_n(1)$. Combining now with Lemma \ref{fn to phi}, the fact that $e^{it\Delta_{x,y}}$ is an isometry on $H_{x,y}^s$ and the asymptotic positivity of the energies of linear profiles deduced in Step 1 we obtain
\begin{align*}
\mH^*(\phi^j)&\leq \frac{1}{2}\|\nabla_{\R^4}\phi^j\|_{L^2(\R^4)}^2=\lim_{n\to\infty}\mH(T_n^j\phi^j)\leq \mH_0<\SSS^2/4.
\end{align*}
Now by Lemma \ref{kenig merle coercive} we see that $\phi^j$ satisfies \eqref{threshold assumption rd} and the desired claim follows.

\subsubsection*{Step 3: Conclusion}
The remaining proof follows essentially the same line as in the proof of \cite[Prop. 7.1]{HaniPausader}, where we should suitably replace the nonlinear estimates applied on $\R\times\T^2$ in \cite{HaniPausader} by the ones applied on $\R^3\times\T$, the latter being established in \cite{RmT1}. Since the adaptation is straightforward and tedious, we shall only give here a sketch of the key arguments and refer to \cite{HaniPausader} and \cite{RmT1} for the full details. Using \eqref{orthog L2} and \eqref{cnls conv of h} we infer that for any finite $1\leq k\leq K^*$
\begin{align}
\mM_0&=\sum_{j=1}^k \mM(T_n^j\tdu^j)+\mM(w_n^k)+o_n(1),\label{mo sum}\\
\mH_0&=\sum_{j=1}^k \mH(T_n^j\tdu^j)+\mH(w_n^k)+o_n(1)\label{eo sum}.
\end{align}
For \eqref{mo sum} and \eqref{eo sum} two different scenarios shall potentially take place: either
\begin{align}
\sup_{j\in\N}\lim_{n\to\infty}\mM(T_n^jP_n^j\phi^j)&=\mM_0\text{ and}\nonumber\\
\sup_{j\in\N}\lim_{n\to\infty}\mH(T_n^jP_n^j\phi^j)&=\mH_0,\label{first situation}
\end{align}
or there exists some $\delta>0$ such that
\begin{align}
\sup_{j\in\N}\lim_{n\to\infty}\mM(T_n^jP_n^j\phi^j)&\leq \mM_0-\delta\text{ or}\nonumber\\
\sup_{j\in\N}\lim_{n\to\infty}\mH(T_n^jP_n^j\phi^j)&\leq\mH_0-\delta.\label{second situation}
\end{align}
In the case \eqref{first situation}, by the asymptotic positivity of the energies of the linear profiles deduced in Step 1 we know that there exists exactly one non-zero linear profile $\phi^1$ and
\begin{align*}
u_n(0)=T_n^1\phi^1+w_n^1.
\end{align*}
Particularly, from \eqref{mo sum} and \eqref{eo sum} it follows
\begin{align}
\lim_{n\to\infty}\mM(T_n^1 \phi^1)&=\mM_0,\\
\lim_{n\to\infty}\mH(T_n^1 \phi^1)&=\mH_0,\\
\lim_{n\to\infty}\|w_n^1\|_2&=0,\label{l2 constraint w1}\\
\lim_{n\to\infty}\mH(w_n^1)&=0.\label{energy constraint w1}
\end{align}
Combining with Lemma \ref{cnls killip visan curve}, \eqref{energy constraint w1} also implies
\begin{align}
\lim_{n\to\infty}\|\nabla_{x,y} w_n^1\|_2=0,
\end{align}
thus together with \eqref{l2 constraint w1} we deduce
\begin{align}
\lim_{n\to\infty}\|w_n^1\|_{H_{x,y}^1}=0\label{h1 constraint w1}.
\end{align}
If $\phi^1$ corresponds to a Euclidean profile (which is Case IIa in the proof of \cite[Prop. 7.1]{HaniPausader}) then we are able to apply Lemma \ref{small scale proxy lem lem} (which is applicable by Step 2) to obtain
\begin{align}\label{contradiction 1}
\limsup_{n\to\infty}\|u_n\|_{Z(I_n)}<\infty,
\end{align}
which contradicts \eqref{precondition}. Thus $\phi^1\in H_{x,y}^1$ corresponds to a scale-one profile (which is Case IIc in the proof of \cite[Prop. 7.1]{HaniPausader}) and
\begin{align}
u_n(0,z)=e^{it_n^1 \Delta_{x,y}}\phi^1(z-p_n^1)+w_n^1.
\end{align}
Notice that since $\T$ is compact, we may simply assume that $p_n^1=(x_n^1,0)$. If $t_n^1\equiv 0$, then we are done. Otherwise $t_n^1\to\pm\infty$. We show that this leads to a contradiction. It suffices to consider the case $t_n^1\to\infty$, the case $t_n^1\to-\infty$ can be dealt similarly. We have
\begin{align*}
\|e^{it\Delta_{x,y}}T_n^1 \phi^1\|_{Z(\inf I_n,0)}\leq \|e^{it\Delta_{x,y}}T_n^1 \phi^1\|_{Z(-\infty,0)}
=\|e^{it\Delta_{x,y}}\phi^1\|_{Z(-\infty,-t_n^1)}\to 0
\end{align*}
as $n\to\infty$. But then by Lemma \ref{lemma small data} we reach the contradiction \eqref{contradiction 1} again. This finishes the discussion of the case \eqref{first situation}. If otherwise case \eqref{second situation} takes place, then Step 1 and Lemma \ref{cnls killip visan curve} imply
\begin{align*}
\sup_{1\leq j\leq K^*}\limsup_{n\to\infty}\mD(T_n^j\phi^j)<\mD^*.
\end{align*}
This is exactly Case III in the proof of \cite[Prop. 7.1]{HaniPausader}. Now arguing as in \cite{HaniPausader}, by the inductive hypothesis \eqref{introductive hypothesis}, the stability Lemma \ref{lem stability cnls} (setting $\tilde{u}=\sum_{j=1}^k u_n^j+e^{it\Delta_{x,y}}w_n^k$ and $u=u_n$ therein), the orthogonality condition \eqref{cnls orthog of pairs} and the smallness condition \eqref{cnls to zero wnk} we arrive at the contradiction \eqref{contradiction 1} again. This completes the desired proof.
\end{proof}

\begin{lemma}[Existence of a minimal blow-up solution]\label{category 0 and 1}
Suppose that $\mD^*\in(0,\infty)$. Then there exists a global solution $u_c$ of \eqref{nls} such that $\mD(u_c)=\mD^*$ and
\begin{align*}
\|u_c\|_{Z((-\infty,0])}=\|u_c\|_{Z([0,\infty))}=\infty.
\end{align*}
Moreover, $u_c$ is almost periodic in $H_{x,y}^1$ modulo $\R_x^3$-translations.
\end{lemma}

\begin{proof}
As discussed at the beginning of this section, under the assumption $\mD^*<\infty$ one can find a sequence $(u_n)_n$ of solutions of \eqref{nls} that satisfies the preconditions of Lemma \ref{Palais Smale}. We apply Lemma \ref{Palais Smale} to infer that $(u_n(0))_n$ (up to modifying time and space translation) is precompact in $H_{x,y}^1$. We denote its strong $H_{x,y}^1$-limit by $\psi$. Let $u_c$ be the solution of \eqref{nls} with $u_c(0)=\psi$. Then $\mD(u_c(t))=\mD(\psi)=\mD^*$ for all $t$ in the maximal lifespan $I_{\max}$ of $u_c$ (recall that $\mD$ is a conserved quantity).

We firstly show that $u_c$ is a global solution. It suffices to show that $s_0:=\sup I_{\max}=\infty$, the negative direction can be similarly proved. If this does not hold, then by Lemma \ref{lemma small data} there exists a sequence $(s_n)_n\subset \R$ with $s_n\to s_0$ such that
\begin{align*}
\lim_{n\to\infty}\|u_c\|_{Z((-\inf I_{\max},s_n])}=\lim_{n\to\infty}\|u_c\|_{Z([s_n,\sup I_{\max}))}=\infty.
\end{align*}
Define $u_n(t):=u_c(t+s_n)$. Then \eqref{precondition} is satisfied with $t_n\equiv 0$. We then apply Lemma \ref{Palais Smale} to the sequence $(u_n(0))_n$ to conclude that there exists some $\varphi\in H_{x,y}^1$ such that, up to modifying the space translation, $u_c(s_n)$ strongly converges to $\varphi$ in $H_{x,y}^1$. But then using Strichartz we obtain
\begin{align*}
\|e^{it\Delta_{x,y}}u_c(s_n)\|_{Z([0,s_0-s_n))}=\|e^{it\Delta_{x,y}}\varphi\|_{Z([0,s_0-s_n))}+o_n(1)=o_n(1).
\end{align*}
By Lemma \ref{lemma small data} we can extend $u_c$ beyond $s_0$, which contradicts the maximality of $s_0$. Now by \eqref{oo1} and Lemma \ref{lem stability cnls} it is necessary that
\begin{align}\label{blow up uc}
\|u_c\|_{Z((-\infty,0])}=\|u_c\|_{Z([0,\infty))}=\infty.
\end{align}

We finally show that the orbit $\{u_c(t):t\in\R\}$ is precompact in $H_{x,y}^1$ modulo $\R^3_x$-translations. Let $(\tau_n)_n\subset\R$ be an arbitrary time sequence. Then \eqref{blow up uc} implies
\begin{align*}
\|u_c\|_{Z((-\infty,\tau_n])}=\|u_c\|_{Z([\tau_n,\infty))}=\infty.
\end{align*}
The claim follows by applying Lemma \ref{Palais Smale} to $(u_c(\tau_n))_n$.
\end{proof}

We end this section by establishing some useful properties of the minimal blow-up solution $u_c$.

\begin{lemma}\label{lemma property of uc}
Let $u_c$ be the minimal blow-up solution given by Lemma \ref{category 0 and 1}. Then
\begin{itemize}
\item[(i)] There exists a center function $x:\R\to \R^3$ such that for each $\vare>0$ there exists $R>0$ such that
\begin{align}
\int_{|x+x(t)|\geq R}|\nabla_{x,y} u_c(t)|^2+|u_c(t)|^2+|u_c(t)|^{4}\,dxdy\leq\vare\quad\forall\,t\in\R.
\end{align}

\item[(ii)]There exists some $\delta>0$ such that $\inf_{t\in\R}\mK(u_c(t))=\delta$.
\end{itemize}
\end{lemma}

\begin{proof}
(i) is an immediate consequence of the inhomogeneous Gagliardo-Nirenberg inequality on $\R^3\times\T$ and the almost periodicity of $u_c$ in $H_{x,y}^1$. We follow the same line as in the proof of \cite[Prop. 10.3]{killip_visan_soliton} to show (ii). Assume the contrary that the claim does not hold. Then we can find a sequence $(t_n)_n\subset\R$ such that $\mK(u_c(t_n))\to 0$. By the almost periodicity of $u_c$ in $H_{x,y}^1$ we can find some $v_0$ such that $u(t_n)$ converges strongly to $v_0$ in $H_{x,y}^1$ (modulo $x$-translations). Combining with the continuity of $\mK$ in $H_{x,y}^1$ we infer that
\begin{align*}
\mD(v_0)&=\mD(u_c)=\mD^*\in(0,\infty),\\
\mK(v_0)&=\lim_{n\to\infty}\mK(u_c(t_n))=0.
\end{align*}
This is however impossible due to Lemma \ref{cnls killip visan curve} (i).
\end{proof}

\subsection{Extinction of the minimal blow-up solution}
In this subsection we close the proof of Theorem \ref{main thm} by showing the contradiction that the minimal blow-up solution $u_c$ must be equal to zero. We shall firstly record a key result concerning the growth of the center function $x(t)$ given in Lemma \ref{lemma property of uc}. The proof of the result relies on the conservation of momentum and was initially given in \cite{non_radial}. Similar results under the framework of MEI-functional and in the setting of NLS on a waveguide were proved in \cite{killip_visan_soliton} and \cite{HaniPausader} respectively. The proof in our case is a straightforward combination of the arguments in \cite{non_radial,killip_visan_soliton,HaniPausader} and we thus omit the details here.

\begin{lemma}\label{holmer}
Let $x(t)$ be the center function given by Lemma \ref{lemma property of uc}. Then $x(t)$ obeys the decay condition $x(t)=o(t)$ as $|t|\to\infty$.
\end{lemma}

We are now ready to prove Theorem \ref{main thm}.

\begin{proof}[Proof of Theorem \ref{main thm}]
We aim to prove the contradiction that the minimal blow-up solution $u_c$ given by Lemma \ref{category 0 and 1} is equal to zero. Let $\chi:\R^3\to\R$ be a smooth radial cut-off function satisfying
\begin{align*}
\chi=\left\{
             \begin{array}{ll}
             |x|^2,&\text{if $|x|\leq 1$},\\
             0,&\text{if $|x|\geq 2$}.
             \end{array}
\right.
\end{align*}
For $R>0$, we define the local virial action $z_R(t)$ by
\begin{align*}
z_{R}(t):=\int R^2\chi\bg(\frac{x}{R}\bg)|u_c(t,x,y)|^2\,dxdy.
\end{align*}
Direct calculation yields
\begin{align}
\pt_t z_R(t)=&\,2\,\mathrm{Im}\int R\nabla_x \chi\bg(\frac{x}{R}\bg)\cdot\nabla_x u_c(t)\bar{u}_c(t)\,dxdy,\label{final4}\\
\pt_t^2 z_R(t)=&\,4\int \pt^2_{x_j x_k}\chi\bg(\frac{x}{R}\bg)\pt_{x_j} u_c\pt_{x_k}\bar{u}_c\,dxdy-\frac{1}{R^2}\int\Delta_x^2\chi\bg(\frac{x}{R}\bg)|u_c|^2\,dxdy\nonumber\\
-&\,\int\Delta_x\chi\bg(\frac{x}{R}\bg)|u_c|^{4}\,dxdy.
\end{align}
We then obtain
\begin{align*}
\pt_t^2 z_R(t)=8\mK(u_c(t))+A_R(u_c(t)),
\end{align*}
where
\begin{align*}
A_R(u_c(t))=&\,4\int\bg(\pt^2_{x_j}\chi\bg(\frac{x}{R}\bg)-2\bg)|\pt_{x_j} u_c|^2\,dxdy+4\sum_{j\neq k}\int_{R\leq|x|\leq 2R}\pt_{x_j}\pt_{x_k}\chi\bg(\frac{x}{R}\bg)\pt_{x_j} u_c\pt_{x_k}\bar{u}_c\,dxdy\nonumber\\
-&\,\frac{1}{R^2}\int\Delta_x^2\chi\bg(\frac{x}{R}\bg)|u_c|^2\,dxdy
-\int\bg(\Delta_x\chi\bg(\frac{x}{R}\bg)-6\bg)|u_c|^{4}\,dxdy.
\end{align*}
We have the rough estimate
\begin{align*}
|A_R(u(t))|\leq C_1\int_{|x|\geq R}|\nabla_x u_c(t)|^2+\frac{1}{R^2}|u_c(t)|^2+|u_c(t)|^{4}\,dxdy
\end{align*}
for some $C_1>0$. By Lemma \ref{lemma property of uc} we know that there exists some $\delta>0$ such that
\begin{align}\label{small extinction ff}
\inf_{t\in\R}(8\mK(u_c(t)))\geq 8\delta=:2\eta_1>0.
\end{align}
From Lemma \ref{lemma property of uc} it also follows that there exists some $R_0\geq 1$ such that
\begin{align*}
\int_{|x+x(t)|\geq R_0}|\nabla_{x,y} u_c(t)|^2+|u_c(t)|^2+|u_c(t)|^{4}\,dxdy\leq \frac{\eta_1}{C_1}.
\end{align*}
Thus for any $R\geq R_0+\sup_{t\in[t_0,t_1]}|x(t)|$ with some to be determined $t_0,t_1\in[0,\infty)$, we have
\begin{align}\label{final3}
\pt_t^2 z_R(t)\geq \eta_1
\end{align}
for all $t\in[t_0,t_1]$. By Lemma \ref{holmer} we know that for any $\eta_2>0$ there exists some $t_0\gg 1$ such that $|x(t)|\leq\eta_2 t$ for all $t\geq t_0$. Now set $R=R_0+\eta_2 t_1$. Integrating \eqref{final3} over $[t_0,t_1]$ yields
\begin{align}\label{12}
\pt_t z_R(t_1)-\pt_t z_R(t_0)\geq \eta_1 (t_1-t_0).
\end{align}
Using \eqref{final4}, Cauchy-Schwarz and Lemma \ref{cnls killip visan curve} we have
\begin{align}\label{13}
|\pt_t z_{R}(t)|\leq C_2 \mD^*R= C_2 \mD^*(R_0+\eta_2 t_1)
\end{align}
for some $C_2=C_2(\mD^*)>0$. \eqref{12} and \eqref{13} give us
\begin{align*}
2C_2 \mD^*(R_0+\eta_2 t_1)\geq\eta_1 (t_1-t_0).
\end{align*}
Setting $\eta_2=\frac{\eta_1}{4C_2\mD^*}$, dividing both sides by $t_1$ and then sending $t_1$ to infinity we obtain $\frac{1}{2}\eta_1\geq\eta_1$, which implies $\eta_1\leq 0$, a contradiction. This completes the proof.
\end{proof}

\subsection{Finite time blow-up below ground states}
In the final subsection we give the proof of Theorem \ref{thm blow up}.

\begin{proof}[Proof of Theorem \ref{thm blow up}]
The proof makes use of the classical Glassey's virial arguments \cite{Glassey1977}. In the context of normalized ground states, we shall invoke the same idea from the proof of \cite[Thm. 1.5]{BellazziniJeanjean2016} to show the claim. We firstly prove the following statement: for $\phi\in H_{x,y}^1$ satisfying $\mH(\phi)<m_{\mM(\phi)}$ and $\mK(\phi)<0$ one has
\begin{align}\label{energy trapping}
\mK(\phi)\leq \mH(\phi)-m_{\mM(\phi)}.
\end{align}
Indeed, from Lemma \ref{monotoneproperty} we know that there exists some $t^*\in(0,1)$ such that $\mK(\phi^{t^*})=0$ and $\frac{d}{ds}(\mH(\phi^{s}))(s)\geq \frac{d}{ds}(\mH(\phi^{s}))(1)=\mK(\phi)$ for $s\in(t^*,1)$. Then
\begin{align*}
\mH(\phi)&=\mH(\phi^1)=\mH(\phi^{t^*})+\int_{t^*}^1\frac{d}{ds}(\mH(\phi^{s}))(s)\,ds\geq\mH(\phi^{t^*})+(1-t^*)\frac{d}{ds}(\mH(\phi^{s}))(1)\nonumber\\
&=\mH(\phi^{t^*})+(1-t^*)\mK(\phi)> m_{\mM(\phi)}+\mK(\phi),
\end{align*}
which implies \eqref{energy trapping}. Next, using the same arguments as in the proof of Lemma \ref{invariance from mA} we infer that $\mK(u(t))<0$ for all $t$ in the maximal lifespan of $u$, thus also $\mK(u(t))\leq m_{\mM(u)}-\mH(u)$. We now define
\begin{align*}
V(t):=\int |x|^2|u(t)|^2\,dxdy.
\end{align*}
By using the same approximation arguments as in the proof of \cite[Prop. 6.5.1]{Cazenave2003} we know that $|x|u(t)\in L_{x,y}^2$ for all $t$ in the maximal lifespan of $u$. Direct calculation (which is similar to the one given in the proof of Theorem \ref{main thm}) yields
\begin{align*}
\pt_{t}^2V(t)=8\mK(u(t))\leq 8(m_{\mM(u)}-\mH(u))<0.
\end{align*}
This particularly implies that $t\mapsto V(t)$ is a positive and concave function simultaneously. Hence the function $t\mapsto V(t)$ can not exist for all $t\in\R$ and the desired claim follows.
\end{proof}

\subsubsection*{Acknowledgements}
The author acknowledges the funding by Deutsche Forschungsgemeinschaft (DFG) through the Priority Programme SPP-1886 (No. NE 21382-1). The author is also grateful to Zehua Zhao for some stimulating discussions.

\addcontentsline{toc}{section}{References}

\end{document}